\documentclass[10pt,a4paper,american]{amsart}
\usepackage[left=3.5cm, right=3.5cm]{geometry}
\usepackage[american]{babel}
\usepackage{amsthm}
\usepackage{amsbsy}
\usepackage{amstext}
\usepackage{amsfonts,amsmath,amsthm,latexsym,amssymb,marvosym,esint,xfrac,bbm,amssymb}
\usepackage{graphics, epsfig}
\usepackage{color}
\usepackage{tikz}
\tikzstyle{mybox} = [draw=black, very thick, rectangle, rounded corners, inner ysep=5pt, inner xsep=5pt]
\usepackage{pgfplots}
\usepackage[shortlabels]{enumitem}
\usepackage{appendix}
\usepackage{amsfonts}
\usepackage{mathcomp}
\usepackage{dsfont}
\usepackage{latexsym}
\usepackage{amssymb}
\usepackage{esint}
\usepackage{stmaryrd}
\usepackage{comment}
\usepackage[colorlinks,pdfpagelabels,pdfstartview = FitH,bookmarksopen = true,bookmarksnumbered = true,linkcolor = blue,plainpages = false, hypertexnames = true,citecolor = red,pagebackref=true]{hyperref}
\usepackage{cite}
\usepackage{etoolbox}

\newtheorem{theo}{Theorem}[section]
\newtheorem{cor}[theo]{Corollary}

\newtheorem{lem}[theo]{Lemma}
\theoremstyle{definition}
\newtheorem{defin}[theo]{Definition}
\newtheorem*{lem*}{Lemma}
\newtheorem{rem}[theo]{Remark}

\newtheorem*{cor*}{Corollary}
\newtheorem*{theo*}{Theorem}

\newcommand{\norm}[1]{\lVert#1\rVert}

\DeclareMathOperator*{\esssup}{ess\,sup}
\DeclareMathOperator*{\essinf}{ess\,inf}

\DeclareMathOperator*{\supp}{supp}
\DeclareMathOperator*{\sgn}{sgn}

\def\XXint#1#2#3{{\setbox0=\hbox{$#1{#2#3}{\int}$ }
\vcenter{\hbox{$#2#3$ }}\kern-.6\wd0}}

\def\Xiint#1{\mathchoice
{\XXiint\displaystyle\textstyle{#1}}%
{\XXiint\textstyle\scriptstyle{#1}}%
{\XXiint\scriptstyle\scriptscriptstyle{#1}}%
{\XXiint\scriptscriptstyle\scriptscriptstyle{#1}}%
\!\iint}
\def\XXiint#1#2#3{{\setbox0=\hbox{$#1{#2#3}{\iint}$ }
\vcenter{\hbox{$#2#3$ }}\kern-.55\wd0}}

\def\dashiint{\Xiint {\rule{6mm}{0.5pt}}}
\renewcommand{\d}{\:\:\!\!\mathrm{d}}
\newcommand{\N}{\ensuremath{\mathbb{N}}}
\newcommand{\R}{\ensuremath{\mathbb{R}}}

\newcommand{\K}{\mathbb{K}}

\newcommand{\A}{\mathcal{A}}
\newcommand{\DD}{\mathcal{D}}

\renewcommand{\b}{\mathfrak{b}}

\numberwithin{equation}{section}
\allowdisplaybreaks

 \makeatletter
\patchcmd{\@setaddresses}{\indent}{\noindent}{}{}
\patchcmd{\@setaddresses}{\indent}{\noindent}{}{}
\patchcmd{\@setaddresses}{\indent}{\noindent}{}{}
\patchcmd{\@setaddresses}{\indent}{\noindent}{}{}
\makeatother

\begin{document}
\renewcommand{\refname}{References}
\renewcommand{\abstractname}{Abstract}

\title[Boundedness and Optimal Support]{Boundedness, ultracontractive bounds and optimal evolution of the support for doubly nonlinear anisotropic diffusion}
\date{\today}

\author[S. Ciani]{Simone Ciani}
\address{Simone Ciani,
Department of Mathematics, University of Bologna Alma Mater,
Piazza di Porta S. Donato, 5, 40126 Bologna, Italy}
\email{simone.ciani3@unibo.it}

\author[V. Vespri]{Vincenzo Vespri}
\address{Vincenzo Vespri,
Università di Firenze Dipartimento di Matematica ed Informatica "Ulisse Dini",
Viale Morgagni 67/a, 50134 Firenze, Italy,
Member of G.N.A.M.P.A. (I.N.d.A.M.) }
\email{vincenzo.vespri@unifi.it}

\author[M. Vestberg]{Matias Vestberg}
\address{Matias Vestberg,
Department of Mathematics, Uppsala University,
P.~O.~Box 480, 751 06, Uppsala, Sweden}
\email{matias.vestberg@math.uu.se}

\begin{abstract}
We investigate some regularity properties of a class of doubly nonlinear anisotropic evolution equations whose model case is 
\begin{align*}
  \partial_t \big(|u|^{\alpha -1}u \big) - \sum^N_{i=1} \partial_i \big( |\partial_i u|^{p_i - 2} \partial_i u \big) = 0, 
  \end{align*} 
where $\alpha \in (0,1)$ and $p_i \in (1, \infty)$. We obtain super and ultracontractive bounds, and global boundedness in space for solutions to the Cauchy problem with initial data in $L^{\alpha+1}(\mathbb{R}^N)$, and show that the mass is nonincreasing over time. As a consequence, compactly supported evolution is shown for optimal exponents. We introduce a seemingly new paradigm, by showing that Caccioppoli estimates, local boundedness and semicontinuity are consequences of the membership to a suitable energy class. This membership is proved by first establishing the continuity of the map $t \mapsto |u|^{\alpha-1}u(\cdot,t) \in L^{1+1/\alpha}_{\textrm{loc}}(\Omega)$ permitting us to use a suitable mollified weak formulation along with an appropriate test function.
\end{abstract}
\maketitle

\setcounter{tocdepth}{1}

\begin{center}
	\begin{minipage}{9.3cm}
		\small
		\tableofcontents
	\end{minipage}
\end{center}

\vskip0.5cm \noindent 
 
\section{Introduction}\label{sec: intro}
\noindent This work is concerned with local and global regularity properties of weak solutions to doubly nonlinear anisotropic evolution equations of the form
\begin{align}\label{eq:diffusion}
\partial_t \big( |u|^{\alpha -1} u\big)  - \nabla\cdot A(x,t,u,\nabla u) = 0 \quad \text{ in } \quad \Omega_T:=\Omega\times (0,T),
\end{align}
where $\Omega \subset \R^N$ is an open bounded set, and $A$ is a Caratheodory vector field satisfying the conditions
\begin{align}
\label{cond:structure1} A(x,t,s, \xi)\cdot \xi &\geq \Lambda^{-1} \sum^N_{i=1}|\xi_i|^{p_i},
 \\
\label{cond:structure2} |A_i(x,t,s, \xi)| &\leq \Lambda \big( \sum^N_{k=1} |\xi_k|^{p_k}\big)^\frac{p_i-1}{p_i},\hspace{5mm}i \in \{1,\dots, N\}.
\end{align}
Conditions of this type for parabolic equations were previously considered in \cite{YuLi}, and are somewhat more general than the following conditions which are also frequently used, see for example \cite{Antontsev-Shamrev}, \cite{DuMoVe}: 
 \begin{align*}
 A_i(x,t,s, \xi)\xi_i &\geq \Lambda^{-1} |\xi_i|^{p_i}, \hspace{5mm}i \in \{1,\dots, N\},
 \\
 |A_i(x,t,s, \xi)| &\leq \Lambda |\xi_i|^{p_i -1},\hspace{5mm}i \in \{1,\dots, N\}.
\end{align*}
The model case for such vector fields is 
\begin{align*}
 A_i(x,s,\xi) = |\xi_i|^{p_i-2}\xi_i,
\end{align*}
so that we get the prototype equation
\begin{align}\label{eq:prototype}
 \partial_t \big(|u|^{\alpha -1}u \big) - \sum^N_{i=1} \partial_i \big( |\partial_i u|^{p_i - 2} \partial_i u \big) = 0,
\end{align}

\noindent The equations considered in this paper combine in a unitary fashion two of the most studied nonlinear equations in the last two decades: the doubly nonlinear equations and the anisotropic ones.\vskip0.2cm 

\noindent Doubly nonlinear equations were introduced in the 60s by Lions  \cite{L} and Kalashnikov  \cite{K}. The term doubly nonlinear refers to the fact that there is a nonlinearity both in the elliptic part and the diffusion part of the equation. Equations of this kind have a broad spectrum of applications in many physical contexts, for instance as the flows of nonhomogeneous non-Newtonian fluids, and simultaneous motion in the surface channel in the underground water, just to name a few. We refer to Chapter 4 of the book \cite{AntShm-Energy} and the references therein for an account of the applications. Even if the theory is quite complete, especially in the degenerate and singular supercritical case (for the exact definition of these technical terms, see for example  \cite{FoHeVe}), some important questions still remain open and are the object of intense research. 
\vskip0.2cm 
\noindent The anisotropic equations were introduced in the 80s by Giaquinta \cite{Gia} and Marcellini \cite{Mar}. The term anisotropic comes from the fact that the diffusion is of the power type which can vary according to the directions. While in the case of equations with differentiable coefficients there are many regularity results (see for example\cite{BoBr} and \cite{EMM}), for what concerns rough coefficients the theory of regularity is still in its infancy. In fact, if the $L^\infty$-estimates and the Critical Mass Lemma, two technical tools necessary to demonstrate regularity (See \ref{est:local_bddness} and \ref{lem:DG-type} later on), can be adapted in this situation, the same does not apply for the so-called shrinking lemma (see \cite{DiBene}, Lemma 7.2 Chap III and Lemma 5.1 Chap IV). We recall that the shrinking lemma, introduced by De Giorgi \cite{DeGi} is the other fundamental technical tool needed to demonstrate regularity. Thanks to this lemma, it can be proved that, while taking a smaller domain, if we denote with $\mu$ the infimum of the solution, the measure of the set where the solution takes on values between $\mu + \varepsilon$  and $\mu$  tends to zero when $\varepsilon$ goes to zero (for more details we refer the reader to\cite{ DiBeUrbVes}. In a pioneering work Liskevich and Skrypinik \cite{LS} succeeded in proving regularity in a very special case by substituting the shrinking lemma with the positivity expansion approach introduced in \cite{DiBeGiVe0} in the parabolic context. Despite these recent results, as already mentioned, the theory of regularity is extremely fragmented. For example, Harnack's estimates for the elliptic operator are known, to our knowledge, only for operators with constant coefficients and with all exponents $p_i$ greater than 2 satisfying a parabolic condition (see \cite{CiaMosVes} for more details).

 \vskip0.2 cm \noindent Finally, the joint nonlinearity of doubly nonlinear anisotropic equations were first investigated in the works \cite{Degtyarev-Tedeev-Bilateral} and \cite{TedDeg}, which proved, inter alia, support growth estimates for the prototypical anisotropic operators and $L^\infty$-estimates. More precisely, in \cite{Degtyarev-Tedeev-Bilateral}, the authors study the compact support for solutions to the prototype equation to \eqref{eq:diffusion}, using a strong notion of solution and an interpolated anisotropic Gagliardo-Nirenberg inequality. Hence, with a proof given in the successive paper \cite{TedDeg}, they study the lifetime of the solution and give local estimates of its $L^{\infty}$-norm in terms of certain integral quantities, showing that the support growth estimates are  optimal.\newline \noindent 
 In this paper, not only we extend their results to operators with non-smooth coefficients (for which the theory of regularity is unknown) but also we prove, in this context, the aforementioned Critical Mass Lemma and other structural results that lay down the foundations of most subsequent work related to the regularity of solutions of this kind of equations.
 
 \subsubsection*{Applications}
 As an application one gets for free sharp estimates on the support of Barenblatt-type solutions (see for instance \cite{CiaMosVes}), which play the role of the fundamental solutions for these operators. Indeed, no explicit self-similar fundamental solutions is known (see \cite{FeVaVo} for a discussion on this topic) and the exponential shift commonly used in\cite{DiGiVe-mono} winds up the spatial anisotropy (see section Novelty and Significance in \cite{CiaSkrVes} for instance).\newline \noindent
 Further, the ultracontractive estimates  that we are about to describe play an important role in the existence of solutions for the Cauchy Problem with $L^1_{loc}(\Omega)$ initial data (see for instance \cite{DiBeHer}).
\vskip0.2cm \noindent 
\subsection*{Plan of the paper}
Concerning local weak solutions to \eqref{eq:diffusion}, we draw a detailed analysis of 
\begin{enumerate}
    \item[1] Mollified weak formulation of the notion of solution, that eventually leads us to the continuity in time of the map \[t \mapsto |u|^{\alpha-1}u \in L^{1+1/\alpha}_{\textnormal{loc}} (\Omega);\]\vskip0.1cm \noindent 
    \item[2] Energy bounds for solutions;
    \item[3] Local boundedness and lower semicontinuity (for the whole energy class).\vskip0.1cm \noindent 
\end{enumerate}
\noindent  On the other hand, regarding weak solutions of Cauchy Problem associated to \eqref{eq:diffusion} we study the following properties
\vskip0.2cm \noindent 
    \begin{enumerate}
        \item[4.1] The evolution of the $L^1(\R^N)$ and $L^{\alpha+1}(\R^N)$ norms;
        \vskip0.1cm \noindent 
        \item[4.2] Ultracontractivity properties;\vskip0.1cm \noindent 
        \item[4.3] Compactly supported evolution.
    \end{enumerate}
\vskip0.2cm \noindent Hereafter we introduce each one of these aspects. While in the literature (see for instance \cite{Degtyarev-Tedeev-Bilateral}) the continuity in time is often included in the definition, here we start by the full variational definition of local weak solutions. Suitably adapting an idea of \cite{St}, we prove the aforementioned continuity in time and give a mollified weak formulation that dispenses with the usual Steklov averaging technique. Finally, this smoothed formulation is used to show that local weak solutions to \eqref{eq:diffusion} are elements of special energy classes.

\subsection*{Energy Classes} In this work we pursue an approach that dates back to De Giorgi (see \cite{DeGi}): we show that certain regularity properties are embodied in general energy estimates rather than in the mere class of solutions to an equation. Here we define two classes of functions $\DD(p_i,\alpha, \Omega_T)$, $\A(p_i,\alpha, \Omega_T)$ (see the end of Section \ref{sec:energy}), that we believe may be of paramount importance in the study of the local behavior of solutions associated to general doubly nonlinear anisotropic operators. The set inclusion $\A(p_i,\alpha, \Omega_T) \subseteq \DD(p_i,\alpha, \Omega_T)$ means, roughly speaking, that classic Caccioppoli estimates are a consequence of testing with a specified Lipschitz function $f(u)$ (see Lemma \ref{lem:Energy_Est} for more details). This fact is clearly not new, but mainly employed in the regularity theory for systems, see for instance (see for instance \cite{DiBene}, chapter VIII). Anisotropic operators as \eqref{eq:diffusion} have many similarities with systems, since each energetic term is independent from the others.\newline \noindent The classes $\A(p_i,\alpha, \Omega_T)$, while being more restrictive, promise wider application for the regularity of anisotropic operators. Indeed, to make an example, let us mention that in the case of the parabolic $p$-Laplace equation, the precise property of expansion of positivity has been found by means of logarithmic estimates (see for instance \cite{DiBene}); these are estimates obtained, loosely speaking, by testing with functions of the form $\ln(H/[H-(1-u)])^+$, $H>0$, hence corresponding to the description of the class $\A(p,1, \Omega_T)$.

\subsection*{Local Boundedness and Semicontinuity}
\noindent To what pertains the boundedness of local weak solutions to \eqref{eq:diffusion}, we are interested in the parameter range
\begin{align}\label{parameter-range}
\alpha \in (0,1), \hspace{7mm} 1 < p_i < \bar p \Big(1+ \frac{\alpha +1}{N}\Big), \hspace{7mm} \bar p < N,
\end{align}
where
\begin{align*}
 \bar p := \Big( \frac1N \sum^N_{i=1} \frac{1}{p_i}\Big)^{-1}.
\end{align*}
The limit case $\alpha=1$ recovers the known facts in \cite{DuMoVe}. More generally, we prove that under certain conditions involving the exponents $p_i$s and $\alpha$ we have \[ \DD(p_i,\alpha, \Omega_T)\subseteq L^{\infty}_{\textnormal{loc}}(\Omega_T).\] \noindent Hence, by the membership of local weak solutions of \eqref{eq:diffusion} to the class $\DD(p_i,\alpha, \Omega)$, we obtain the local boundedness for weak solutions.\newline
We need to distinguish between the cases $\bar p > \frac{N(\alpha + 1)}{N + \alpha +1}$ and $\bar p \leq \frac{N(\alpha + 1)}{N + \alpha +1}$. In the latter case we need the extra integrability condition 
\begin{align}\label{extra_integrability}
 u \in L_\textnormal{loc}^m(\Omega_T), \textnormal{ for some } m > \frac{N}{\bar p}(\alpha +1 - \bar p),
\end{align} similarly to sub-critical $p$-Laplacian equations (see for instance \cite{DiBene}). We remark that in the case $\alpha=1$, which corresponds to the usual anisotropic equations, the two ranges for $\bar p$ are identical to those appearing in \cite{YuLi}, and also the extra integrability condition reduces to the integrability assumption used in \cite{YuLi}.
\vskip0.1cm \noindent Moreover, we show that elements of the class $\DD(p_i, \alpha, \Omega_T)$ are lower-semicontinuous. The strategy adopted in \cite{DuMoVe} adapted the idea of \cite{Kuusi} to the anisotropic metric induced by the equation. In the doubly nonlinear scenario complications arise since the difference between a solution and a constant is not necessarily another solution. This first difficulty was faced in \cite{AvLu}, leaving however open the question whether a similar result could be found for sign-changing solutions. Here we follow the very general method of \cite{Naian}, whose prerogative is that lower-semicontinuity is shown to be a consequence of some kind of measure-theoretical maximum principle, referred to in the literature as a Critical Mass Lemma, or a De Giorgi-type Lemma.

\subsection*{Ultracontractivity and the evolution of the $L^1(\R^N)$ and $L^{\alpha+1}(\R^N)$ norms} We consider the Cauchy Problem \begin{align}\label{CP}
 \left\{
\begin{array}{ll}
\partial_t \big( |u|^{\alpha -1} u\big)  - \nabla\cdot A(x,t,u,\nabla u) = 0, & \quad \text{in } S_T:=\R^N \times (0,T), 
\\[5pt]
 u(x,0) = u_0(x),  & \quad x \in \R^N,
\end{array}
\right.
\end{align} with $A$ being the Caratheodory field associated to \eqref{eq:diffusion}-\eqref{cond:structure1}-\eqref{cond:structure2}. For a local weak solution $u\in \cap_{i=1}^N {L^{p_i}(S_T)}$ of \eqref{CP}, we consider the representative for which $|u|^{\alpha-1}u$ is continuous with respect to time. Then, the mass is nonincreasing, meaning that 
\begin{align*}
 \norm{u(\cdot,t)}_{L^1(\R^N)}\leq \norm{u_0}_{L^1(\R^N)},
\end{align*} 
and a corresponding estimate holds also for the $L^{\alpha+1}$-norm. When $\lambda_1:= N(\bar p - (\alpha+1)) + \bar p > 0$ we have the following ultra-contractivity property, 
 \begin{align}\label{Intro-est:LinfL1}
 \norm{u(\cdot, \tau)}_{L^\infty(\R^N)} \leq c \tau^{-\frac{N}{\lambda_1}} \Big( \int_{\R^N} |u_0| \d x \Big)^\frac{\bar p}{\lambda_1}, \qquad \forall \tau \in (0,T).
 \end{align} 
\noindent The problem has been addressed in \cite{DuMoVe} for the case of anisotropic $p$-Laplacian operators and earlier in \cite{TedDeg} following the approach of \cite{DiBeHer} with a simplifying technique again by \cite{AndTed} and obtaining the contractivity properties above by means of the quantities \[|||u|||_r= \sup_{\rho \ge r} \rho^{-\mu}\int_{E_{\rho}} |u(x)|\d x,\]  
where $\mu>0$ is a constant and $E_\rho$ is a suitably scaled rectangle depending on the parameter $\rho$. Here we point out that in the framework of nonlinear semigroup theory these $L^1$-$L^{\infty}$ estimates are usually referred to as ultracontractive bounds. In the linear case, it is well-known that estimates of this type are equivalent to specific Sobolev inequalities for the Dirichlet form that is associated to the generator of the evolution under consideration. A different but particularly interesting approach to these estimates is given in \cite{BonGri}, where the authors prove the $L^1$-$L^{\infty}$ estimate above for the isotropic case, by exploiting a logarithmic Sobolev inequality.

\subsection*{Finite Speed Propagation}
In the context of non-Newtonian fluids, already in the 80s some techniques were known in order to estimate the support of solutions by comparison (see for instance \cite{DiHe}, \cite{Wata}) or by energy methods (see \cite{Antontsev}). A further step in doubly nonlinear equations was made in \cite{AntDia}, where the energy methods were fully exploited to avoid the comparison principle, which is not available in this case. For a more complete picture see Chapter 3 of \cite{AntShm-Energy} and its bibliographical remarks. In our case, we exploit a particular choice of test functions that, being compactly supported away from the initial datum, allow recursive iterative inequalities à la De Giorgi. Nevertheless, this approach was already introduced in \cite{DiHe} with the construction of suitable supersolutions; while the optimal behaviour of the support of solutions of doubly nonlinear degenerate equations has been investigated in \cite{Degtyarev-Tedeev-Bilateral} (see also \cite{TeVe} for the case of systems), adapting the technique of \cite{AndTed} originally conceived for high-order equations. Along this line, compactly supported anisotropic evolution was studied in \cite{DuMoVe} (see also \cite{AAloc}), exploiting the technique of bounding all the energy by a precise choice of test functions allowing the estimate for a single direction (see Section \ref{sec: supp} for more details). Here we take advantage of this technique to face the double nonlinearity.\vskip0.1cm \noindent 
Despite doubly nonlinear equations being the natural bridge between two classic nonlinear parabolic equations, namely the $p$-Laplacian and the porous medium equation (see on this topic, \cite{DiBeUrbVes},\cite{Vazquez}), the precise property of finite speed of propagation  is reminiscent of hyperbolic equations (see for instance \cite{Aronson} for the porous medium equation). In this scenario, a common point is again the energy method, which avoids the comparison principle (see for instance the book \cite{Antontsev-Shamrev}) and is useful to determine the qualitative property of finite speed of propagation, intended as the property of dead cores formation. On a different path, we start from a nontrivial compactly supported datum $u_0$ and we study the evolution in time of the support of the solution to the Cauchy problem associated with \eqref{eq:diffusion}.

\begin{theo} 
Suppose that the condition 
\begin{equation}\label{Intro-slow diffusion}
\alpha+1<p_i \leq p_N < \bar{p} ( 1+ \alpha/N)< N + \alpha
\end{equation} is satisfied for all $i = 1,\dots, N$. Let $u\in \bigcap_{i=1}^N L^{p_i}(S_T)$ be local weak solution to the Cauchy problem \eqref{CP} and suppose
\[ u_0 \in L^{1+\alpha}(\R^N)\cap L^1(\R^N), \quad \quad \emptyset \ne \text{supp}(u_0) \subset [-R_0,R_0]^N=: \K_{R_0}.\] Then the support of $u$ evolves with the law
\begin{equation} \label{intro-supporto}
\text{supp}(u(\cdot, t)) \subset \prod_{i=1}^N [-R_i(t), R_i(t)], \quad R_i(t)= 2R_0+ \gamma \|u_0\|_{L^1(\R^N)}^{\frac{\bar{p}(p_i-\alpha-1)}{\lambda_1 p_i}} t^{\frac{N(\bar{p}-p_i) + \bar{p}}{\lambda_1p_i}}.
\end{equation}

\end{theo}

\subsection*{Optimality of the estimates} Under our general assumptions on the structure of the operator $A$, we can prove that both the estimates on the $L^{\infty}$-norm of solutions and on the support that we provide here are optimal. More precisely, let $S_T$ be the strip $S_T= \R^N \times (0,T)$ and $K$ a cube in $\R^N$ of side length $R_0$. Now, solutions $u\in \cap L^{p_i}(S_T)$ to the Cauchy problem \eqref{CP} that have integrable initial datum $u_0\in L^{\alpha+1}(\R^N)$ which is compactly supported in $K$, satisfy under the condition $\lambda_1:= N(\bar p - (\alpha+1)) + \bar p > 0$ the estimate
 \[
 \norm{u(\cdot, \tau)}_{L^\infty(\R^N)} \leq c \tau^{-\frac{N}{\lambda_1}} \Big( \int_{\R^N} |u_0| \d x \Big)^\frac{\bar p}{\lambda_1},
 \]
for all $\tau \in (0,T)$. If \eqref{Intro-slow diffusion} holds true, the support of $u$ evolves with the law
\[\text{supp}(u(\cdot, t)) \subset \prod_{i=1}^N [-R_i(t), R_i(t)], \quad R_i(t)= 2R_0+ \gamma \|u_0\|_{L^1(\R^N)}^{\frac{\bar{p}(p_i-\alpha-1)}{\lambda_1 p_i}} t^{\frac{N(\bar{p}-p_i) + \bar{p}}{\lambda_1p_i}}.\]
\noindent These estimates are shown to be optimal by the following rectangular computation:
\[||u(\cdot, t)||_{L^1(\R^N)}\leq ||u(\cdot, t)||_{L^{\infty}(\R^N)} |\supp(u(\cdot, t))|\leq \gamma ||u_0||_{L^1(\R^N)},  \qquad \gamma>0,\, \,  t>>2R_0,\]
because 
\[|\supp(u(\cdot, t))| \leq  (2\gamma)^N ||u_0||_{L^1(\R^N)}^{\sum_{i=1}^N \frac{\bar{p}(p_i-\alpha-1)}{(\lambda_1 p_i)}} t^{\sum_{i=1}^N\frac{N(\bar{p}-p_i)+\bar{p}}{\lambda_1 p_i}} = \gamma ||u_0||_{L^{1}(\R^N)}^{\frac{(\lambda_1-\bar{p})}{\lambda_1}} t^{N/\lambda_1} .\]
 If one of the two estimates would have been non-optimal, then the fact that the mass is nonincreasing would be contradicted.

\subsection{Structure of the paper} In Section \ref{sec:setting} we define the notion of solution, the associated function spaces and we introduce the tools of the trade: exponential mollification, monotonicity inequalities and the main embeddings. In Section \ref{sec:time-cont} we show the continuity in time of $|u|^{\alpha-1}u$ as a map $[0,T] \rightarrow L^{(\alpha+1)/\alpha}_{\textnormal{loc}}(\Omega)$ and we give a more handy  definition of solution that involves time derivatives. In Section \ref{sec:energy} we derive the main energy estimates and accordingly we define the functional classes $\A(p_i, \alpha, \Omega_T)$ and $\DD(p_i, \alpha, \Omega_T)$. Then in Section \ref{sec: localboundedness} we study the local boundedness of functions belonging to the function class $\DD(p_i, \alpha, \Omega_T)$ and in Section \ref{sec: semicontinuity} we study their pointwise behaviour. Lastly, in Section \ref{sec: Cauchy} we give precise estimates of the evolution of the $L^{\infty}$ norm of the solutions to the Cauchy problem, we study the time evolution of their $L^1$ and $L^{\alpha+1}$ norms. Finally, in Section \ref{sec: supp} we estimate the evolution of their support.

\vspace{10mm}
\noindent
{\bf Acknowledgments.} This work was partially supported by the Wallenberg AI, Autonomous Systems and Software Program (WASP) funded by the Knut and Alice Wallenberg Foundation. Simone Ciani expresses gratitude to both Departments of Mathematics of the Technical University of Darmstadt and the University of Bologna, as former and current places of employment. Specifically, the current place of employment is supported itself by PNR 2021-2027 fundings of MIUR, that we acknowledge. Vincenzo Vespri wants to express his gratitude towards GNAMPA (INdAM).

\vspace{10mm}

\section{Notation} \label{sec: notation}

{\small \begin{itemize}
    \item Referring to the main structure in the Introduction, we will consider a vector of real numbers ${\bf{p}}=(p_1, \dots, p_N)$ and we assume without loss of generality that $p_1 \leq \dots p_N.$
    The harmonic mean $\bar{p}$ is defined as $\bar{p}= (\sum_{i=1}^N 1/p_i)^{-1}$ and for $\bar{p}<N$ its Sobolev conjugate $\bar{p}^*= N\bar{p}/(N-\bar{p})$. Finally for $\alpha\in (0,1)$ as in the Introduction,  we set $P = \max\{ (\alpha +1), \,p_N\}$.

\vskip0.2cm \noindent 

\item We define \[\bar{p}_{\sigma}=\bar{p} (1+\sigma/N).\]

\vskip0.2cm \noindent 
\item For a number $n\in \N$, a set $E\subset \R^n$ and a vector $v\in \R^n$, we denote as usual the set 
$ v + E = \{ v + x \,|\, x \in E\}$. For every $r>0$ we define the the following $N$-dimensional hyper-rectangles:
\begin{align*}
 K_r &:= (-r^\frac{1}{p_1},r^\frac{1}{p_1}) \times \dots \times (-r^\frac{1}{p_N}, r^\frac{1}{p_N}),
 \\
 \K_r &:= [-r,r]^N.
\end{align*}
 Moreover, for $(x_o,t_o) \in \R^N\times \R$ and $r>0$ we define the space-time cylinders
\begin{align*}
 Q_r(x_o,t_o) := (x_o,t_o) + K_r \times (-r, 0].
\end{align*} 
\vskip0.2cm \noindent 
\item For $\gamma > 0$ and $a\in \R$ we understand $|a|^{\gamma-1}a$ to be zero if $a=0$ even though in this case technically the first factor is ill-defined if $\gamma < 1$. We occasionally denote $a^\gamma = |a|^{\gamma -1} a$ to simplify the notation.
\vskip0.2cm \noindent 

\item Let $\alpha \in (0,1)$, $\beta := 1/\alpha > 1$ and define for $w,v \in \R$ the quantities
\vskip0.1cm \noindent 
\begin{align*}
\notag
\b[v,w] :=& \tfrac{1}{\beta+1} (|v|^{\beta+1}-|w|^{\beta+1}) - |w|^{\beta-1}w (v-w)
\\
\notag =& \tfrac{\beta}{\beta+1} (|w|^{\beta+1}-|v|^{\beta+1})- v (|w|^{\beta-1}w-|v|^{\beta-1}v), 
\\ \notag
\\ 
\qquad \b_\alpha[v,w] :=& \b[|v|^{\alpha-1}v,|w|^{\alpha-1}w] = \tfrac{\alpha}{\alpha+1} ( |v|^{\alpha+1} - |w|^{\alpha+1}) - w (|v|^{\alpha-1}v - |w|^{\alpha-1}w).
\end{align*}
The quantity $\b[v,w]$ for nonnegative $v,w$ was used in \cite{SiVe,SiVe2,VeVe}, with a notation consistent with the one presented above. We alert the reader that there is a slight notational discrepancy with the work \cite{BoeDuKoSc}, in which a signed vectorial version of the quantity is used.

\vskip0.2cm \noindent 
\item Constants along the estimates may vary from line to line, when no dependence on the solution or other important iterative quantities is embodied.

\end{itemize}
}

\newpage

\section{Setting and Preliminaries}\label{sec:setting}
Here we introduce some notation and present auxiliary tools that will be useful in the course of the paper. We start by the definition of weak solutions to \eqref{eq:diffusion}, and to this aim we briefly recall the definition of the anisotropic Sobolev spaces. Given a vector of numbers ${\bf p} = (p_1,\dots, p_N)$ as in Section \ref{sec: notation} with $p_i > 1$ we set 
\begin{align*}
 W^{1, {\bf p}}_\textrm{o}(\Omega) &:= \{ v \in W^{1,1}_\textrm{o}(\Omega)\,|\, \partial_i v \in L^{p_i}(\Omega)\}.
 \\
 W^{1, {\bf p}}_\textrm{loc}(\Omega) &:= \{ v \in W^{1,1}_\textrm{loc}(\Omega)\,|\, \partial_i v \in L^{p_i}_\textrm{loc}(\Omega)\},
\end{align*}
and
\begin{align*}
 L^{\bf p}(0,T; W^{1, {\bf p}}(\Omega)) &:= \{ v \in L^1(0,T;W^{1,1}(\Omega)) \,|\, \partial_i v \in L^{p_i}(\Omega_T)\},
 \\
 L^{\bf p}_\textrm{loc}(0,T; W^{1, {\bf p}}_\textrm{loc}(\Omega)) &:= \{ v \in L^1_\textrm{loc}(0,T;W^{1,1}_\textrm{loc}(\Omega))\,|\, \partial_i v \in L^{p_i}_\textrm{loc}(\Omega_T)\},
 \\
 L^{\bf p}(0,T; W^{1, {\bf p}}_\textrm{loc}(\Omega)) &:= \{ v \in L^1(0,T;W^{1,1}_\textrm{loc}(\Omega))\,|\, \partial_i v \in L^{p_i}(0,T; L^{p_i}_\textrm{loc}(\Omega))\}.
\end{align*}

\begin{defin}\label{def:weaksol}
A function $u \in L^{\bf p}(0,T; W^{1, {\bf p}}_\textrm{loc}(\Omega)) \cap L^{P}(0,T; L^P_\textrm{loc}(\Omega))$ where 

\noindent $P = \max\{ (\alpha +1), \,p_N\}$ is a solution to \eqref{eq:diffusion} if 
\begin{align}\label{eq:weak_form}
&\iint_{\Omega_T} A(x,t,u,\nabla u)\cdot \nabla \varphi- |u|^{\alpha-1} u\partial_t \varphi\d x\d t=0,
\end{align}
for all $\varphi \in C^\infty_o(\Omega_T)$.
\end{defin}

We remark that one could define solutions in an analogous way and obtain local regularity results also for weak solutions in the larger space $L^{\bf p}_\textrm{loc}(0,T; W^{1, {\bf p}}_\textrm{loc}(\Omega)) \cap L^P_\textrm{loc}(\Omega_T)$. However, in order to prove the desired time continuity on the whole interval $[0,T]$ i.e. including time zero, it seems necessary to have global integrability properties in time. The arguments are also clearer when this type of integrablility is assumed. The local regularity results obtained under our assumptions also hold for solutions in the larger space since a simple translation in time brings about the situation we consider here.

\subsection{Auxiliary tools}
We now recall some elementary lemmas that will be used later, and start by defining a mollification in time as in \cite{KiLi}, see also \cite{BoeDuMa}. For $T>0$, $t\in [0,T]$, $h\in (0,T)$ and $v\in L^1(\Omega_T)$ we set
\begin{align}
\label{def:moll}
v_h(x,t):=\frac{1}{h}\int^t_0 e^\frac{s-t}{h}v(x,s)\d s.
\end{align}
Moreover, we define the reversed analogue by
\begin{align*}
v_{\overline h}(x,t) :=\frac{1}{h}\int^T_t e^\frac{t-s}{h}v(x,s)\d s.
\end{align*}
For details regarding the properties of the exponential mollification we refer to \cite[Lemma 2.2]{KiLi}, \cite[Lemma 2.2]{BoeDuMa}, \cite[Lemma 2.9]{St}. The properties of the mollification that we will use have been collected for convenience into the following lemma:
\begin{lem}
\label{expmolproperties} Suppose that $v \in L^1(\Omega_T)$, and let $p\in[1,\infty)$. Then the mollification $v_h$ defined in \eqref{def:moll} has the following properties:
\begin{enumerate}
\item[(i)]
If $v\in L^p(\Omega_T)$ then $v_h\in L^p(\Omega_T)$,
$$
\norm{v_h}_{L^p(\Omega_T)}\leq \norm{v}_{L^p(\Omega_T)},
$$
and $v_h\to v$ in $L^p(\Omega_T)$. A similar estimate also holds with $v_{\bar h}$ on the left-hand side.
\item[(ii)]
In the above situation, $v_h$ has a weak time derivative $\partial_t v_h$ on $\Omega_T$ given by
\begin{align*}
\partial_t v_h=\tfrac{1}{h}(v-v_h),
\end{align*}
whereas for $v_{\overline h}$ we have
\begin{align*}
\partial_t v_{\overline h}=\tfrac{1}{h}(v_{\overline h}-v).
\end{align*}
\item[(iii)]
If $v$ has a weak partial derivative in space then so does $v_h$ and $v_{\bar h}$ and 
\begin{align*}
 \partial_j (v_h) = (\partial_j v)_h, \hspace{5mm} \partial_j (v_{\bar h}) = (\partial_j v)_{\bar h}.
\end{align*}
\item[(iv)] If $v\in L^p(0,T;L^{p}(\Omega))$ then $v_h, v_{\bar h} \in C([0,T];L^{p}(\Omega))$.
\end{enumerate}
\end{lem}
 
\begin{rem}\label{rem:expmol_local_integ}
 The exponential time mollification \eqref{def:moll} is well defined also if the function $v$ is only in $L^1(0,T;L^1_\textrm{loc}(\Omega))$ and the properties $(i)$ and $(iv)$ in the previous lemma evidently have corresponding versions for functions that are only locally integrable in space.
\end{rem}

The next Lemma provides us with some useful estimates for the quantity $\b_\alpha[v,w]$ that was defined in Section \ref{sec: notation}. The estimates can be derived from the corresponding properties for the quantity $\b$ which were proved in \cite[Lemma 2.3]{BoeDuKoSc}. Note again the difference in notation between our work and \cite{BoeDuKoSc}.
\begin{lem}
\label{estimates:boundary_terms}
Let $v,w \in \R$ and $\alpha \in(0,1)$. Then there exists a constant $c$ depending only on $\alpha$ such that:
\begin{enumerate}
\item[(i)] $\tfrac 1 c\big| |w|^{\frac{\alpha-1}{2}}w - |v|^{\frac{\alpha-1}{2}}v \big|^2 \leq \b_\alpha[v,w] \leq c \big| |w|^{\frac{\alpha-1}{2}}w - |v|^{\frac{\alpha-1}{2}}v \big|^2$, \vspace{2mm}
\item[(ii)] $\tfrac 1 c (|w|+|v|)^{\alpha-1}|w - v|^2 \leq  \b_\alpha[v,w] \leq c (|w|+|v|)^{\alpha-1} |w - v|^2 $ , 
\item[(iii)]$\b_\alpha[v,w] \leq c |v - w|^{1+\alpha}$.
\end{enumerate}
\end{lem}
The following observation regarding real numbers is used frequently in our calculations.
\begin{lem}\label{lem:elementary_real}
 Let $\gamma > 1$. For all $a, b \in \R$ we have 
 \begin{align}\label{est:exponent_inside}
  |a-b|^\gamma \leq c\big||a|^{\gamma-1} a - |b|^{\gamma-1}b\big|.
 \end{align}
for a constant $c=c(\gamma)$. 
\end{lem}
\begin{proof}[Proof]
 By symmetry we may assume that $a\geq b$. In the case $a,b\geq 0$ the triangular inequality of the $\ell^\gamma$-norm in $\R^2$ applied to $(a-b,0)$ and $(0,b)$ implies that 
 \begin{align*}
  ((a-b)^\gamma + b^\gamma)^\frac{1}{\gamma} \leq a,
 \end{align*}
from which \eqref{est:exponent_inside} follows with $c=1$. The case in which $a,b\leq 0$ follows from the previous case. It remains to consider the case $a\geq 0 > b$. Then we can write
\begin{align*}
 |a-b|^\gamma = |a+|b||^\gamma \leq 2^{\gamma-1}(a^\gamma + |b|^\gamma)= 2^{\gamma-1}|a^\gamma - |b|^{\gamma-1}b|,
\end{align*}
which is \eqref{est:exponent_inside} with $c=2^{\gamma-1}$.
\end{proof}
\noindent Next, we recall some useful anisotropic Sobolev embeddings, both for space and time. 
\begin{lem}
[Sobolev-Troisi embedding,\cite{DuMoVe}]\label{SOBE}
Let $\Omega\subseteq \R^N$ be a rectangular domain, $\bar p<N$ and $\alpha_{i}>0$, $i=1, \dots, N$. If we define 
\[
p_{\alpha}^{*}=\bar p^{*}\frac{\tilde{\alpha}}{N},\qquad \tilde{\alpha}=\sum_{i=1}^{N}\alpha_{i},
\] \noindent then there exists a constant $C=C(N, {\bf p}, {\bf \alpha})>0$ such that for all $u\in W^{1, {\bf p}}(\Omega)$ it holds
\begin{equation}
\label{ST}
\|u\|_{L^{p^{*}_{\alpha}}(\Omega)}\leq C\prod_{i=1}^{N}\|\partial_i|u|^{\alpha_{i}}\|_{L^{p_{i}}(\Omega)}^{\frac{1}{\tilde{\alpha}}}.
\end{equation}
\end{lem}
\noindent As a corollary, one gets the usual Troisi inequality (see \cite{Tro}).

\begin{rem}\label{rem:M_Troisi}
Let $\Omega \, \subseteq  \R^N$ be a rectangular domain. Then there is a constant $C=C(N, {\bf p})>0$ such that 
\begin{align}\label{est:Troisi}
 \norm{u}_{L^{\bar p^*}(\Omega)} \leq C \prod^N_{i=1} \norm{\partial_i u}_{L^{p_i}(\Omega)}^\frac1N,
\end{align}
for all $u \in W^{1, {\bf p}}_\textnormal{o}(\Omega)$. Moreover, taking both sides of \eqref{est:Troisi} to the exponent $\bar p$ and using Young's inequality on the right-hand side with the exponents $\frac{N p_i}{\bar p}$, whose inverses add up to $1$ due to the definition of $\bar p$, we obtain the following useful inequality,
\begin{align}\label{est:Troisi-application}
 \Big(\int_\Omega |u|^{\bar p^*}\d x\Big)^\frac{\bar p}{\bar p^*} \leq C \sum^N_{i=1} \int_\Omega |\partial_i u|^{p_i} \d x, \quad u \in W^{1, {\bf p}}_\textnormal{o}(\Omega).
\end{align}

\end{rem}

\noindent The following is a parabolic anisotropic Sobolev embedding, which can be found in \cite{DuMoVe}.
\begin{theo}
\label{PAS}
Let $\Omega\, \subseteq \,  \R^N$ be a rectangular domain, $\bar p<N$, $\alpha_{i}>0$ for $i=1,\dots, N$ and $\sigma\in [1, p_{\alpha}^{*}]$. For any number $\theta\in [0,\, \bar p/ \bar p^{*}]$ define 
\[
q=q(\theta, {\bf p}, {\bf \alpha})=\theta \,  p^{*}_{\alpha}+\sigma\, (1-\theta).
\]
Then for any $u \in L^{1}(0, T; W^{1,1}_{0}(\Omega))$, there exists a constant $C=C(N, {\bf p}, {\bf \alpha}, \theta, \sigma)>0$ such that
\begin{equation}
\label{PS}
\iint_{\Omega_{T}}|u|^{q} \d x \d t\leq C\, T^{1-\theta\, \frac{\bar p^{*}}{\bar p}}\left(\sup_{t\in [0, T]}\int_{\Omega}|u|^{\sigma}(x, t) \d x\right)^{1-\theta}\prod_{i=1}^{N}\left(\iint_{\Omega_{T}}|\partial_i|u|^{\alpha_{i}}|^{p_{i}}  \d x \d t\right)^{\frac{\theta\,  \bar p^{*}}{N \, p_{i}}}.
\end{equation} 

\end{theo}

\noindent Next we present some variants of the Lemma of Fast Convergence. We first state and prove a rather general version, which implies the two other variants that we will utilize in the De Giorgi type iterations. 

\begin{lem}\label{lem:fastconfg-new} Let $0<\mu \leq \nu$ and let $C>0$ and $b>1$. Suppose that $(Z_n)$ is a sequence of nonnegative numbers satisfying
\begin{align}\label{cond:convergence_of_Zn}
 Z_{n+1} \leq C b^n \max\{Z_n^{1+\mu}, Z_n^{1+\nu}\}, \textnormal{ and } Z_0 \leq \min\{ C^{-\frac{1}{\mu}}, C^{-\frac{1}{\nu}} \} b^{-\frac{1}{\mu^2}}.
\end{align}
Then 
\begin{align}\label{conclusion:convergence}
 Z_n \leq b^{-\frac{n}{\mu}}Z_0, \textnormal{ and thus } \lim_{n\to\infty} Z_n = 0.
\end{align}
\end{lem}
\begin{proof}{}
 We prove the inequality in \eqref{conclusion:convergence} by induction. The case $n=0$ is obvious. Suppose that the inequality holds for some $n$. Then by the induction assumption and \eqref{cond:convergence_of_Zn} we have 
 \begin{align}\label{est:Z_nplusone-long-calc}
  \notag Z_{n+1} \leq C b^n \max\{ (b^{-\frac{n}{\mu}}Z_0)^{1+\mu}, (b^{-\frac{n}{\mu}}Z_0)^{1+\nu}\} &= C b^n \max\{ b^{-\frac{n}{\mu}-n}Z_0^{\mu}, b^{-\frac{n}{\mu}-n\frac{\nu}{\mu}}Z_0^\nu\} Z_0 
  \\
  &\leq C b^{-\frac{n}{\mu}} \max\{Z_0^{\mu},Z_0^{\nu} \}Z_0,
 \end{align}
where in the last step we also used that $\nu/\mu \geq 1$. Using the second estimate in \eqref{cond:convergence_of_Zn} we see that 
\begin{align*}
 \max\{Z_0^{\mu},Z_0^{\nu}\} \leq C^{-1} b^{-\frac{1}{\mu}}.
\end{align*}
Combining this estimate with \eqref{est:Z_nplusone-long-calc} completes the induction.
\end{proof}
\noindent The following consequence of the previous lemma will be useful.
\begin{lem}\label{lem:fastconvg-general}
Let $\chi_{i}>0$ for $i=1, \dots, N$, and let $C>0$ and $b>1$. Suppose that $(Z_n)$ is a sequence of nonnegative numbers satisfying
\begin{equation} \label{iterHP}
Z_{n+1}\leq C\, b^{n}\, \frac{1}{N}\sum_{i=1}^{N}Z_{n}^{1+\chi_{i}}.
\end{equation} 
Denote $\chi_{\min} =\min\{\chi_{1}, \dots, \chi_{N}\}$ and $\chi_{\max}  =\max\{\chi_{1}, \dots, \chi_{N}\}$. If
\begin{align*}
Z_0 \leq \min\{ C^{-\frac{1}{\chi_{\min}}}, C^{-\frac{1}{\chi_{\max}}} \} b^{-\frac{1}{\chi_{\min}^2}},
\end{align*}
then 
\[
 Z_n \leq b^{-\frac{n}{\chi_{\min}}}Z_0 \textnormal{ and thus } \lim_{n\to \infty }Z_{n}= 0.
\]
\end{lem}
\begin{proof}{}
 The result follows directly from the previous lemma since \eqref{iterHP} implies
 \begin{align*}
  Z_{n+1} \leq C b^n \max\{Z_n^{1+\chi_{\min}}, Z_n^{1+\chi_{\max}}\}
 \end{align*}
\end{proof}

\noindent In the case that $\mu=\nu$, Lemma \ref{lem:fastconfg-new} directly implies the following result which will be used frequently. A more direct proof is available in
\cite[Lemma 7.1]{Gi}.
\begin{lem}\label{lem:fastconvg}
Let $(Y_j)^\infty_{j=0}$ be a sequence of nonnegative numbers such that
\begin{equation*}
Y_{j+1}\leq C b^j Y^{1+\delta}_j,
\end{equation*}
where $b >1$ and $C, \delta>0$. If
\begin{equation*}
Y_0\leq C^{-\frac{1}{\delta}}b^{-\frac{1}{\delta^2}},
\end{equation*}
then $(Y_j)$ converges to zero as $j\to\infty$.
\end{lem}

\section{Continuity in time and mollified weak formulation}\label{sec:time-cont}
In this subsection we show that $|u|^{\alpha-1}u$ is continuous in time as a map into $L^{\frac1\alpha+1}_{\textrm{loc}}(\Omega)$. The proof is adapted from \cite{St}. We start with a lemma.
\begin{lem}\label{lem:time-cont}
Suppose that $u$ is a weak solution in the sense of Definition \ref{def:weaksol} and define 
\begin{align*}
\mathcal{V}:=\big\{ w\in L^P(\Omega_T)\, |\, w \in L^{\bf p}_\textnormal{loc}(0,T; W^{1, {\bf p}}_\textnormal{loc}(\Omega)), \, \partial_t w \in L^{\alpha+1}(0,T; L^{\alpha+1}_\textnormal{loc}(\Omega)) \big\}.
\end{align*}
Then, for every $\zeta\in C^\infty_o(\Omega_T,\R_{\geq 0})$ and $w\in \mathcal V$ we have
\begin{align}
\label{eq:time_1}
\iint_{\Omega_T} \partial_t \zeta \b_\alpha[u,w] \d x\d t = \iint_{\Omega_T} A(x,t,u,\nabla u) \cdot \nabla[\zeta(u - w)] + \zeta(|u|^{\alpha - 1} u - |w|^{\alpha -1} w )\partial_t w \d x\d t.
\end{align}
\end{lem}
\begin{proof}[Proof]
Let $w\in \mathcal V$, $\zeta \in C^\infty_o(\Omega_T,\R_{\geq 0})$ and choose
$$
\varphi =\zeta \left( w - u_h \right)
$$
as test function in \eqref{eq:weak_form}. The local $L^P$-integrability of $u$ and $w$ together with the $p_i$-integrability of each $\partial_i u$ and $\partial_i w$ and the structure conditions of the vector field $A$ guarantee that the test function can be justified by approximation with smooth compactly supported test functions. Our goal is to pass to the limit $h\to 0$. It follows from Lemma \ref{expmolproperties}, Remark \ref{rem:expmol_local_integ}  and the aforementioned integrability properties that
\begin{align*}
\iint_{\Omega_T}A(x,t,u,\nabla u)\cdot \nabla \varphi \d t \d t \xrightarrow[h\to 0]{} \iint_{\Omega_T} A(x,t,u,\nabla u)\cdot \nabla [\zeta (w - u)]\d x \d t.
\end{align*} Note that Lemma \ref{expmolproperties} (ii) implies
$$
\big( |u_h|^{\alpha -1}u_h -|u|^{\alpha - 1} u\big) \partial_t u_h\leq 0,
$$
which shows that we can treat the parabolic part as follows.
\begin{align*}
\iint_{\Omega_T} |u|^{\alpha - 1}u \partial_t \varphi \d x\d t &= \iint_{\Omega_T} \zeta |u|^{\alpha - 1}u \partial_t w \d x\d t -\iint_{\Omega_T} \zeta |u_h|^{\alpha -1}u_h \partial_t u_h \d x\d t
\\
&\quad+ \iint_{\Omega_T} \zeta \big( |u_h|^{\alpha -1} u_h - |u|^{\alpha -1}u \big) \partial_t u_h \d x\d t+ \iint_{\Omega_T} \partial_t \zeta |u|^{\alpha -1}u ( w - u_h ) \d x\d t
\\
&\leq \iint_{\Omega_T} \zeta |u|^{\alpha - 1}u \partial_t w \d x\d t + \iint_{\Omega_T} \tfrac{1}{\alpha+1} \partial_t \zeta |u_h|^{\alpha + 1} \d x\d t
\\
&\quad +\iint_{\Omega_T} \partial_t \zeta |u|^{\alpha -1}u ( w - u_h ) \d x\d t
\\
&\xrightarrow[h\to 0]{} \iint_{\Omega_T} \zeta |u|^{\alpha - 1}u \partial_t w \d x\d t + \iint_{\Omega_T} \partial_t \zeta \big(\tfrac{1}{\alpha+1} |u|^{\alpha + 1} + |u|^{\alpha -1} u(w - u) \big) \d x\d t
\\
&\quad =\iint_{\Omega_T} \zeta  (|u|^{\alpha -1}u - |w|^{\alpha -1}w) \partial_t w\d x\d t -\iint_{\Omega_T} \partial_t \zeta \b_\alpha[u,w] \d x \d t,
\end{align*}
This shows ``$\leq$'' in \eqref{eq:time_1}. The reverse inequality can be derived in the same way by taking
$$
\varphi =\zeta \left( w-[u]_{\overline h} \right)
$$
as test function. \end{proof}

\begin{theo}\label{cont_into_Lbetaplusone}
Let $u$ be a weak solution in the sense of Definition \ref{def:weaksol}. Then
\\
$|u|^{\alpha-1}u\in C([0,T];L^{\frac1\alpha + 1}_\textnormal{loc}(\Omega))$.
\end{theo}
\begin{proof}[Proof]
We prove continuity on the interval $[0,\tfrac12 T]$ and describe later how the argument can be modified to show continuity also on $[\tfrac12 T,T]$, thus completing the proof. We first note that due to Lemma \ref{expmolproperties}, $w:= u_{\bar{h}}$ belongs to the set of admissible comparison functions $\mathcal{V}$ of Lemma \ref{lem:time-cont}. Furthermore, Lemma \ref{expmolproperties} (iv) and Remark \ref{rem:expmol_local_integ} guarantee that $w$ is continuous $[0,T]\to L^{\alpha + 1}_\textrm{loc}(\Omega)$. Since $1/\alpha >1$ we obtain from Lemma \ref{lem:elementary_real} that
\begin{align*}
||w(x,s)|^{\alpha-1} w(x,s) - |w(x,t)|^{\alpha - 1}w(x,t)|^{\frac1\alpha + 1}\leq |w(x,s)-w(x,t)|^{\alpha + 1},
\end{align*}
which shows that $|w|^{\alpha -1} w$ is continuous $[0,T]\to L^{\frac1\alpha+1}_\textrm{loc}(\Omega)$. We will show that $|u|^{\alpha -1} u$ is essentially the uniform limit on the time interval $[0,\tfrac12 T]$ of the functions $|w|^{\alpha -1}w$ as $h\to 0$, and the continuity will follow from this.
For a compact set $K\subset \Omega$ we take $\eta \in C^\infty_o(\Omega;[0,1])$ such that $\eta=1$ on $K$ and $|\nabla \eta|\leq C_K$. Furthermore, take $\psi\in C^\infty([0,T];[0,1])$ with $\psi=1$ on $[0,\tfrac12 T]$, $\psi=0$ on $[\tfrac34 T,T]$ and $|\psi'|\leq \tfrac8T$. For $\tau\in (0,\tfrac12 T)$ and $\varepsilon>0$ so small that $\tau+\varepsilon< \tfrac12 T$ we define
\begin{align*}
\chi^\tau_\varepsilon(t)=
\begin{cases}
0, & t<\tau \\
\varepsilon^{-1}(t-\tau), & t\in [\tau, \tau+\varepsilon] \\
1, & t> \tau+\varepsilon.
\end{cases}
\end{align*}
We use \eqref{eq:time_1} with $\zeta=\eta\chi^\tau_\varepsilon\psi$ and $w=u_{\bar{h}}$ to obtain
\begin{align*}
\varepsilon^{-1} \int^{\tau+\varepsilon}_\tau\int_\Omega \b_\alpha[u,u_{\bar{h}}]\eta \d x\d t &= \iint_{\Omega_T} A(x,t,u,\nabla u) \cdot \nabla[\eta(u - u_{\bar{h}})]\chi^\tau_\varepsilon\psi \d x \d t
\\
&\quad + \iint_{\Omega_T}\eta\chi^\tau_\varepsilon\psi (|u|^{\alpha-1}u - |u_{\bar{h}}|^{\alpha-1}u_{\bar{h}})\partial_t u_{\bar{h}} \d x\d t - \iint_{\Omega_T} \b_\alpha[u,u_{\bar{h}}]\eta \psi'\d x \d t
\\
&\leq \sum^N_{i=1} \iint_{\supp \eta \times (0,T)}  |A_i(x,t,u,\nabla u)| (|\partial_i u - (\partial_i u)_{\bar{h}}| + |\partial_i \eta| |u - u_{\bar{h}}|)\d x \d t
\\
& \quad + \frac{8}{T}\iint_{\supp \eta \times (\frac12 T,\frac34 T)} \b_\alpha[u,u_{\bar{h}}]\d x \d t.
\end{align*}
Here we were able to drop the term involving $\partial_t w $ since Lemma \ref{expmolproperties} (ii) shows that the factors $\partial_t w$ and $(|u|^{\alpha-1}u - |w|^{\alpha-1}w)$ are of opposite sign, and hence their product is nonpositive. Passing to the limit $\varepsilon\to 0$ we see that
\begin{align}\label{gs}
\int_K \b_\alpha [u,u_{\bar{h}}](x,\tau) \d x & \leq C_K \sum^N_{i=1} \iint_{\supp \eta \times (0,T)}  |A_i(x,t,u,\nabla u)| (|\partial_i u - (\partial_i u)_{\bar{h}}| + |u - u_{\bar{h}}|)\d x \d t
\notag
\\
 & \quad + \frac{8}{T}\iint_{\supp \eta \times (\frac12 T,\frac34 T)} \b_\alpha[u,u_{\bar{h}}]\d x \d t
\end{align}
for all $\tau\in [0,\tfrac12 T]\setminus N_h$, where $N_h$ is a set of measure zero. Note that the integrand on the left-hand side can be estimated using Lemma \ref{estimates:boundary_terms} (i) and the fact that $\frac{\alpha+1}{2\alpha} > 1$ as follows:
\begin{align*}
||u|^{\alpha -1}u - |u_{\bar{h}}|^{\alpha -1}u_{\bar{h}}|^{\frac1\alpha + 1} = ||u|^{\alpha -1}u - |u_{\bar{h}}|^{\alpha -1}u_{\bar{h}}|^{2\frac{(\alpha + 1)}{2\alpha} } \leq c\big||u|^\frac{\alpha-1}{2}u - |u_{\bar{h}}|^\frac{\alpha-1}{2}u_{\bar{h}} |^2 \leq c \b_\alpha[u,u_{\bar{h}}].
\end{align*}
For the term on the last line of \eqref{gs} we can use Lemma \ref{estimates:boundary_terms} (iii) to make the estimate
\begin{align*}
\b_\alpha[u,u_{\bar{h}}] &\leq c \big|u - u_{\bar{h}} \big|^{1+\alpha} =c |u - u_{\bar{h}}|^\alpha |u-u_{\bar{h}} |
\leq c(|u|^\alpha + |u_{\bar{h}}|^\alpha)|u - u_{\bar{h}}|.
\end{align*}
The first factor stays bounded in $L^{\frac{\alpha + 1}{\alpha}}$ as $h\to 0$ and the second factor converges to zero in $L^{\alpha + 1}$ as $h\to 0$. The structure conditions of $A$, the integrability properties of $u$ and Lemma \ref{expmolproperties} (i) and (iii) show that also the first integral on the right-hand side of \eqref{gs} converges to zero as $h\to 0$. Picking now a sequence $h_j\to 0$ and $w_j= u_{\bar{h}_j}$ and $N:= \cup N_{h_j}$ (which has measure zero) we see that \eqref{gs} combined with the previous observations implies
\begin{align}\label{unif_limit}
\lim_{j\to \infty} \sup_{\tau\in [0,\frac{1}{2}T]\setminus N} \int_K ||u|^{\alpha -1}u -|w_j|^{\alpha - 1} w_j |^{\frac1\alpha + 1}(x,\tau) \d x = 0.
\end{align}
As noted earlier, each $|w_j|^{\alpha-1} w_j$ is continuous as a map $[0,T]\to L^{\frac1\alpha + 1}(K)$. This fact together with the uniform limit \eqref{unif_limit} on the dense set $[0,\frac12T]\setminus N$ and the completeness of $L^{\frac1\alpha + 1}(K)$ show that $|w_j|^{\alpha -1}w_j$ converges uniformly on $[0, \frac12T]$ to a limit function which is continuous into $L^{\frac1\alpha + 1}(K)$. Due to \eqref{unif_limit} this limit is a representative of $|u|^{\alpha - 1}u$.
The continuity on $[\tfrac12T,T]$ follows from a similar argument with $w=u_h$ and with $\psi$ and $\chi^\tau_\varepsilon$ mirrored on the interval $[0,T]$ under the map $t\mapsto T-t$.\end{proof}

\vskip0.2cm \noindent 
Now that we have established the continuity in time it is possible to show that weak solutions satisfy a mollified weak formulation.
\begin{lem}\label{lem:mollified}
Let $u$ be a weak solution in the sense of Definition \ref{def:weaksol}. Then we have
\begin{align}\label{h-averaged-form}
\iint_{\Omega_T} [A(x,\cdot,u, \nabla u)]_h\cdot \nabla \phi+\partial_t [|u|^{\alpha -1}u]_h \phi\d x\d t- \int_\Omega |u|^{\alpha -1}u \phi_{\bar{h}}(x,0) \d x = 0 
\end{align}
for all $\phi \in C^\infty(\Omega \times [0,T])$ with support contained in $K\times [0,\tau]$ ,where $K\subset \Omega$ is compact and $\tau \in (0,T)$. Here $u(x,0)$ refers to the value at time zero of the continuous representative of $|u|^{\alpha -1}u$ as a map $[0,T]\to L^{\frac1\alpha+1}(K)$.
\end{lem}
\begin{proof}[Proof]
Let $\phi$ be as in the statement of the Lemma. Consider the piecewise smooth function
\begin{align*}
\eta_\varepsilon (t):=
\begin{cases}
\, \frac{t}{\varepsilon}, & t\in[0,\varepsilon]
\\
\, 1, & t\in (\varepsilon, T],
\end{cases}
\end{align*}
and use \eqref{eq:weak_form} with the test function $\varphi=\eta_\varepsilon\phi_{\bar{h}}$. Taking the limit $\varepsilon \to 0$ and using Fubini's theorem we see that the elliptic term will converge to the integral of $[A(x,\cdot,u,\nabla u)]_h \cdot \nabla \phi$. Note now that
\begin{align*}
\iint_{\Omega_T} |u|^{\alpha -1} u \partial_t
(\eta_\varepsilon \phi_{\bar{h}})\d x \d t = \iint_{\Omega_T}|u|^{\alpha -1}u \eta_\varepsilon\frac{(\phi_{\bar{h}}-\phi)}{h}\d x \d t +\varepsilon^{-1} \int^\varepsilon_0\int_\Omega |u|^{\alpha -1}u \phi_{\bar{h}} \d x \d t.
\end{align*}
In the first term we can pass to the limit $\varepsilon\to 0$, use Fubini's theorem to move the mollification over to $u$ and apply Lemma \ref{expmolproperties} (ii) to obtain the integral of $\partial_t[|u|^{\alpha -1} u]_h \phi$. It remains to investigate what happens to the last term in the limit $\varepsilon\to 0$. Note that we can write this term as
\begin{align*}
\varepsilon^{-1} \int^\varepsilon_0\int_K |u|^{\alpha-1}u \phi_{\bar{h}} \d x \d t
&=\varepsilon^{-1} \int^\varepsilon_0 \int_K (|u|^{\alpha-1}u)(x,t) \phi_{\bar{h}}(0) \d x \d t 
\\
&\quad + \varepsilon^{-1} \int^\varepsilon_0 \int_K (|u|^{\alpha-1}u)(x,t)[ \phi_{\bar{h}}(t)- \phi_{\bar{h}}(0) ]\d x \d t.
\end{align*}
The second term on the right-hand side converges to zero since $\phi_{\bar{h}}$ is uniformly continuous and $\norm{|u|^{\alpha -1}u(t)}_{L^{\frac1\alpha+1}(K)}$ is bounded independently of $t$. The first term on the right-hand side converges to the second integral on the left-hand side of \eqref{h-averaged-form} since $|u|^{\alpha -1} u \in C([0,T]; L^{\frac1\alpha + 1}(K))$ and $\phi_{\bar{h}}(0)\in L^{\alpha + 1}(\Omega)$.
\end{proof}

\section{Energy Classes}\label{sec:energy}
In this section we prove some useful and very general energy estimates. In particular, the local boundedness and pointwise properties of solutions to \eqref{eq:weak_form} are encoded in the functional class $\A({p_i}, \Omega_T)$ consisting of integrable functions satisfying \eqref{est:energy}.


\begin{lem} \label{lem:general-formula}
Let $\Omega \subset \R^N$ open set, possibly unbounded, and let $T>0$. Let $f:\R \rightarrow \R$ be increasing, Lipschitz and piecewise $C^1$. Suppose that 
\begin{align} \label{Assunzione}
 f(s)= 0 \textnormal{ whenever } f'(s) = 0. 
\end{align} 
  Let $g:\R\to \R$ be the unique function for which
 \begin{align*}
  f(s)=g(|s|^{\alpha-1}s),
 \end{align*}
 and let $G:\R\to \R$ be any integral function of $g$. 
 Let $u$ be a weak solution of \eqref{def:weaksol} in $\Omega_T$. Then for each $\eta\in C^\infty_o(\Omega)$ of the form  
\begin{align}\label{expr:eta}
 \eta(x) := \prod^N_{i=1} \eta_s(x_i)^{p_i}, \quad \eta_i \in C^\infty_o(\R, [0,1]),\quad s \in \{1,\dots N\},
\end{align} 
and $\varphi \in C^{\infty}([0,T];[0,\infty))$, we have for all $0 \leq \tau_1\leq \tau_2\leq T$ the estimate
\begin{equation}\label{general-formula}
\begin{aligned}
\int_{\Omega}& \eta\varphi\, G\big( |u|^{\alpha-1}u(x, \tau_2)\big)\d x + \frac1\gamma \iint_{\Omega\times [\tau_1,\tau_2]} \sum_{i=1}^N |\partial_i u  |^{p_i} f'(u) \eta \varphi \, \d x \d t 
\\
&\leq \gamma \iint_{\Omega\times [\tau_1,\tau_2]} \chi_{\{\nabla u \neq \bar 0\}}\sum_{i=1}^N |f(u)|^{p_i} f'(u)^{1-p_i} |\partial_i \eta^{1/p_i}|^{p_i} \varphi\, \d x \d t 
\\
&\quad  + \int_{\Omega} \eta \varphi \, G\big( |u|^{\alpha-1}u(x, \tau_1)\big) \d x + \iint_{\Omega\times [\tau_1,\tau_2]}  \eta \varphi'  G\big( |u|^{\alpha-1}u \big) \d x \, \d t,
\end{aligned}
\end{equation}
where $\gamma$ is a sufficiently large constant depending only on $\Lambda, N, {\bf p}$. 
\end{lem}

\begin{proof}[Proof]
 Consider first $0<\tau_1 < \tau_2 < T$. Test the mollified weak formulation \eqref{h-averaged-form} with 
 \[\phi=f(u(x,t)) \eta(x) \xi(t),
 \] where $\eta$ is as in \eqref{expr:eta} and  $\xi(t)= \varphi(t)\psi(t)$, being $\varphi$ as above while $\psi$ is the trapezoidal function 
 \begin{align*}
\psi(t)=
\begin{cases}
0, & t<\tau_1 \\
\delta^{-1}(t-\tau_1), & t\in [\tau_1, \tau_1+\delta] \\
1, & t \in (\tau_1+\delta, \tau_2-\delta) \\
1 - \delta^{-1}(t-\tau_2+\delta), & t \in [\tau_2-\delta, \tau_2] \\
0, & t \geq \tau_2.
\end{cases}
\end{align*} Then, owing to \eqref{Assunzione}, the chain rule for weak derivatives shows that $\phi$ can be used in the mollified weak formulation which reads
\begin{equation}\label{molweak:special}
    \begin{aligned}
\iint_{\Omega_T} [A(x,\cdot,u, \nabla u)&]_h\cdot \nabla \phi + \partial_t [|u|^{\alpha -1}u]_h \phi \d x\d t\\
&- \int_{0}^{T} \int_{\Omega} |u_0|^{\alpha -1}u_0 f(u) \eta(x)\xi(t) \tfrac1h e^{-\frac{t}{h}} \d x \d t= 0. 
\end{aligned} \end{equation}

Next we want to perform some estimates in the diffusion term of \eqref{molweak:special} so that we can pass to the limit $h\to 0$. 
Noting that $f(u)= g(|u|^{\alpha-1}u)$ that $g$ is increasing and using (ii) in Lemma \ref{expmolproperties}, we have 
\begin{align}\label{est:diffterm}
 \partial_t [|u|^{\alpha -1}u]_h \phi & = \frac{(|u|^{\alpha-1}u - [|u|^{\alpha-1}u]_h)}{h}(g(|u|^{\alpha-1}u) - g([|u|^{\alpha-1}u]_h) )\eta \xi  
 \\
 \notag & \quad  + \eta \xi  g([|u|^{\alpha-1}u]_h) \partial_t [|u|^{\alpha -1}u]_h
 \\
 \notag & \geq \eta \xi \partial_t \big(G([|u|^{\alpha -1}u]_h)\big),
\end{align}
where in the last step we have also made use of a chain rule for Sobolev functions. The validity of the chain rule in this particular case can be verified by approximating $[|u|^{\alpha -1}u]_h$ using convolutions with standard mollifiers, integrating by parts and using the classical chain rule, and passing to the limit. Combining \eqref{molweak:special} and \eqref{est:diffterm} and moving the time derivative to the test function we end up with
\begin{align*}
 \iint_{\Omega_T} [A(x,\cdot,u, \nabla u)&]_h\cdot \nabla \phi - \eta \partial_t\xi  G([|u|^{\alpha -1}u]_h)  \d x\d t
 \\
&- \int_{0}^{T} \int_{\Omega} |u_0|^{\alpha -1}u_0 f(u) \eta(x)\xi(t) \tfrac1h e^{-\frac{t}{h}} \d x \d t \leq 0.
\end{align*}
We now want to pass to the limit $h\to 0$. The last integral will vanish in this limit due to the dominated convergence theorem. This can be seen by noting that $|u_0|^{\alpha -1}u_0 f(u)$ is locally integrable and the term $\xi(t) \tfrac1h e^{-\frac{t}{h}}$ remains bounded independently of $h$ since $\xi$ vanishes for times less than $\tau_1$. In the other terms passing to the limit $h\to 0$ poses no problem and we recover the corresponding terms without time mollification. Thus we have
\begin{align}\label{est:sidewinder}
 \iint_{\Omega_T} A(x,t,u, \nabla u) \cdot \nabla \phi - \eta \partial_t\xi  G(|u|^{\alpha -1}u)  \d x\d t \leq 0.
\end{align}
In order to estimate the integral of the elliptic term we first note that we can exclude the set of points where $f'\circ u$ is ill-defined. Namely, there is an at most countable set $S\subset \R$ where $f'$ is not defined. Since $\nabla u$ vanishes almost everywhere on any level set we see that $\nabla u$ vanishes almost everywhere on $u^{-1}S$. Due to \eqref{cond:structure2} we have that also $A(x,t,u,\nabla u)$ vanishes almost everywhere on $u^{-1}S$. Denoting $E= \Omega_T \setminus u^{-1}S$ we thus have
\begin{align}\label{eq:change-domain-of-integ}
 \iint_{\Omega_T} A(x,t,u, \nabla u) \cdot \nabla \phi \d x \d t = \iint_E A(x,t,u, \nabla u) \cdot \nabla \phi \d x \d t 
\end{align}
For any point in $E$ where $f'\circ u \neq 0$ we can use the structure conditions \eqref{cond:structure1} and \eqref{cond:structure2} and Young's inequality to make the estimates 
\begin{align*}
 A(&x,t,u, \nabla u) \cdot \nabla \phi = \xi \eta f'(u) A(x,t,u, \nabla u)\cdot \nabla u +\xi f(u) A(x,t,u, \nabla u) \cdot \nabla \eta
 \\
 &\geq \xi \Lambda^{-1}\sum^N_{j=1} |\partial_j u|^{p_j}f'(u) \eta - \Lambda \xi \sum^N_{k=1} \Big( \sum^N_{j=1} |\partial_j u|^{p_j}\Big)^\frac{p_k-1}{p_k} |f(u)| |\partial_k \eta|
 \\
 &\geq \xi \Lambda^{-1} \sum^N_{j=1} |\partial_j u|^{p_j}f'(u) \eta - c \xi \sum^N_{k=1} \Big( \sum^N_{j=1} |\partial_j u|^{p_j}\Big)^\frac{p_k-1}{p_k} |f(u)| \eta_k^{p_k-1}(x_k)|\eta_k'(x_k)| \prod_{s\neq k} \eta_s(x_s)^{p_s}
 \\
 &\geq (2\Lambda)^{-1}\xi \sum^N_{j=1} |\partial_j u|^{p_j}f'(u) \eta - c(\Lambda, N, {\bf p})\chi_{\{\nabla u \neq \bar 0\}} \xi \sum^N_{k=1} |f(u)|^{p_k} f'(u)^{1-p_k} |\partial_k \eta^\frac{1}{p_k}|^{p_k}.
\end{align*}
At the points of $E$ where $f'\circ u = 0$ we obtain the same estimate as above, however without the terms containing $f(u)$, due to \eqref{Assunzione}. Thus, choosing $\gamma= \max\{2\Lambda, c(\Lambda, N, {\bf p})\}$ we have at every point of $E$ that 
\begin{align*}
 \frac1\gamma \xi \sum^N_{j=1} |\partial_j u|^{p_j}f'(u) \eta \leq A(x,t,u, \nabla u) \cdot \nabla \phi + \gamma \chi_{\{\nabla u \neq \bar 0\}} \xi \sum^N_{k=1} |f(u)|^{p_k} f'(u)^{1-p_k} |\partial_k \eta^\frac{1}{p_k}|^{p_k},
\end{align*}
where the last term is interpreted as zero whenever $f\circ u = 0$. Note that at the points of $E$ where $f\circ u \neq 0$, the last term is well-defined due to \eqref{Assunzione}. The first two terms of the estimate are integrable, and the last term is nonnegative, so we can integrate over $E$ to obtain
\begin{align}\notag 
\frac1\gamma \iint_E \xi \sum^N_{j=1} |\partial_j u|^{p_j}f'(u) \eta \d x \d t &\leq \iint_E A(x,t,u, \nabla u) \cdot \nabla \phi \d x \d t 
\\
\label{est:elliptic-integrated} &\quad + \gamma  \iint_{E\cap \{f\circ u \neq 0\}} \chi_{\{\nabla u \neq \bar 0\}} \xi \sum^N_{k=1} |f(u)|^{p_k} f'(u)^{1-p_k} |\partial_k \eta^\frac{1}{p_k}|^{p_k} \d x \d t
\end{align}
where potentially the last integral could be infinite. In the first integral we can exchange $E$ for $\Omega_T$ with the understanding that $\nabla u$ vanishes a.e. on the set where $f'(u)$ is ill-defined so that the integrand can be interpreted as zero on this set. Similarly, for the last integral we can replace $E\cap \{f\circ u\neq 0\}$ by $\Omega_T$ since apart from a set of measure zero, the integrand is well-defined when $\nabla u \neq 0$ and $f\circ u \neq 0$.
With these modifications we combine \eqref{est:elliptic-integrated} with \eqref{eq:change-domain-of-integ} and \eqref{est:sidewinder} to obtain
\begin{align}\notag
 \frac1\gamma \iint_{\Omega_T} \xi \sum^N_{j=1} |\partial_j u|^{p_j}f'(u) \eta \d x \d t &\leq \gamma  \iint_{\Omega_T} \chi_{\{\nabla u \neq \bar 0\}} \xi \sum^N_{k=1} |f(u)|^{p_k} f'(u)^{1-p_k} |\partial_k \eta^\frac{1}{p_k}|^{p_k} \d x \d t
 \\ \label{est:before-delta-to-0}
 & \quad + \iint_{\Omega_T} \eta \partial_t\xi  G(|u|^{\alpha -1}u)  \d x\d t.
\end{align}
Note that the function $\xi=\varphi\psi$ converges pointwise from below to $\varphi \chi_{(\tau_1,\tau_2)}$ as $\delta\to 0$. Thus, the monotone convergence theorem allows us to pass to the limit $\delta \to 0$ in the first two integrals of \eqref{est:before-delta-to-0}, taking the limit inside the integrals obtaining the corresponding integrals in \eqref{general-formula}. In order to treat the last integral in \eqref{est:before-delta-to-0} we note that by the definition of $\xi$ we have
\begin{align}\notag
 \iint_{\Omega_T} \eta \partial_t\xi  G(|u|^{\alpha -1}u)  \d x\d t &=  \iint_{\Omega_T} \eta \varphi' \psi G(|u|^{\alpha -1}u)  \d x\d t + \frac{1}{\delta} \int^{\tau_1+\delta}_{\tau_1} \int_\Omega \eta \varphi G(|u|^{\alpha -1}u)  \d x\d t 
 \\ \label{eq:integral-expanded}
 & \quad - \frac{1}{\delta} \int^{\tau_2}_{\tau_2-\delta} \int_\Omega \eta \varphi G(|u|^{\alpha -1}u)  \d x\d t.
\end{align}
In the first integral we can take the limit $\delta \to 0$ inside the integral due to the dominated convergence theorem, and $\psi$ gets replaced by $\chi_{(\tau_1,\tau_2)}$ thus obtaining another term appearing in \eqref{general-formula}. The continuity property of $|u|^{\alpha-1}u$ established in Lemma \ref{lem:time-cont} can be used to prove that the map $t\mapsto \eta G(|u|^{\alpha-1}u)(\cdot, t)$ is continuous into $L^1(\Omega)$. From this it follows that 
\begin{align*}
 \frac{1}{\delta} \int^{\tau_1+\delta}_{\tau_1} \int_\Omega \eta \varphi G(|u|^{\alpha -1}u)  \d x\d t \xrightarrow[\delta \to 0]{} \int_\Omega \eta \varphi G(|u|^{\alpha -1}u)(x,\tau_1)  \d x.
\end{align*}
The last integral in \eqref{eq:integral-expanded} can be treated in a similar manner. Thus, passing to the limit $\delta \to 0$ in \eqref{est:before-delta-to-0} we recover \eqref{general-formula}. We proved the result in the case $0<\tau_1<\tau_2<T$. The result in the full range $0\leq \tau_1 < \tau_2 \leq T$ now follows from the continuity properties of $|u|^{\alpha-1}u$ and the appropriate convergence theorems by considering the limits $\tau_1 \to 0$ and $\tau_2 \to T$. 
\end{proof} 
When the function $f$ appearing in the previous lemma is chosen appropriately, we recover the classical energy estimates.
\begin{lem}\label{lem:Energy_Est}
Let $u$ be a weak solution in the sense of Definition \ref{def:weaksol}.  Let $k \in \R$, $\eta \in C_o^{\infty}(\Omega; [0,1])$ as in \eqref{expr:eta} and $\varphi \in C^\infty(\R; \R_{\geq 0})$ be a function vanishing near the origin. Then we have

\begin{align}\label{est:energy}
\sum^N_{j=1} &\iint_{\Omega_T} |\partial_j [(u-k)_\pm \eta]|^{p_j} \varphi \d x \d t + \sup_{\tau\in [0,T]}\int_{\Omega\times \{\tau\}} \big(u^\frac{\alpha+1}{2}-k^\frac{\alpha+1}{2})_\pm^2 \eta\varphi \d x
\\
\notag &\leq \sum^N_{j=1} C\iint_{\Omega_T} (u-k)_\pm^{p_j}|\partial_j \eta^\frac{1}{p_j}|^{p_j} \varphi\d x\d t + C\iint_{\Omega_T} \big(u^\frac{\alpha+1}{2}-k^\frac{\alpha+1}{2})_\pm^2 \eta (\partial_t \varphi)_+\d x \d t,
\end{align}
for a constant $C$ depending only on the parameters $\alpha, N, {\bf p}$ and the constant $\Lambda$ appearing in the structure conditions (\ref{cond:structure1}--\ref{cond:structure2}).
\end{lem}
\begin{proof}[Proof]
We prove the lemma in the case of the positive part and comment in the end how the proof can be modified to treat the negative part. We apply Lemma \ref{lem:general-formula} with the function
\begin{align*}
 f(s)=(s-k)_+,
\end{align*}
which clearly satisfies the assumptions of the lemma. With this choice of $f$ we have 
\begin{align*}
 f'(s)=\chi_{\{s>k\}}, \qquad g(s)=f(|s|^{\frac{1}{\alpha}-1}s)= (|s|^{\frac{1}{\alpha}-1}s-k)_+, 
\end{align*}
and we moreover choose $G$ as the integral function
\begin{align*}
G(\tau) := \int^\tau_{k^\alpha} g(s)\d s.
\end{align*}
A simple calculation verifies that 
\begin{align*}
 \qquad G(|u|^{\alpha-1}u)= \b_{\alpha}[u,k]\chi_{\{u>k\}}.
\end{align*}
Therefore, \eqref{general-formula} in this case takes the form 
\begin{equation}\label{ticcio}
\begin{aligned}
\int_{\Omega\times\{\tau_2\}}& \eta\varphi \b_\alpha[u,k]\chi_{\{u>k\}} \d x + \frac1\gamma \int_{\tau_1}^{\tau_2} \int_{\Omega} \sum_{i=1}^N |\partial_i u  |^{p_i} \chi_{\{u>k\}} \eta \varphi \, \d x \d t 
\\
&\leq \gamma \int_{\tau_1}^{\tau_2} \int_{\Omega} \sum_{i=1}^N (u-k)_+^{p_i} |\partial_i \eta^\frac{1}{p_i}|^{p_i} \varphi \d x \d t  + \int_{\Omega\times \{\tau_1 \}} \eta \varphi\b_\alpha[u,k]\chi_{\{u>k\}}\d x
\\
&\qquad \qquad + \int_{\tau_1}^{\tau_2} \int_{\Omega} \varphi' \b_\alpha[u,k]\chi_{\{u>k\}} \d x \, \d t.
\end{aligned}
\end{equation}

\noindent We let $\tau_1\downarrow 0$ and the second integral on the right vanishes because of the properties of $\varphi$, the integrability of $u$ and dominated convergence. Noting that $0\leq \eta_i \leq 1$ we estimate 
\begin{align*}
 |\partial_i[(u-k)_+\eta]|^{p_i} &\leq c|\partial_i(u-k)_+|^{p_i} \eta^{p_i} + c (u-k)_+^{p_i} |\partial_i \eta|^{p_i}
 \\
 &\leq c|\partial_i(u-k)_+|^{p_i}\eta + c (u-k)_+^{p_i}\eta_i(x_i)^{p_i(p_i-1)} |\eta_i'(x_i)|^{p_i}\prod_{s\neq i} \eta_s(x_s)^{p_s p_i}
 \\
 &\leq c|\partial_i u|^{p_i}\chi_{\{u>k\}}\eta + c(u-k)_+^{p_i}|\partial_i \eta^\frac{1}{p_i}|^{p_i}.
\end{align*}
Combining the last estimate with \eqref{ticcio} and noting that we may replace $\partial_t \varphi$ with its positive part we end up with
\begin{align*} 
 \sum^N_{i=1}&\iint_{\Omega_\tau} |\partial_i [(u-k)_+\eta]|^{p_i} \varphi \d x \d t + \int_{\Omega \times \{\tau\}} \b_\alpha[u,k]\chi_{\{ u>k \}}\eta \varphi \d x
\\
\notag &\leq c \sum^N_{i=1}\iint_{\Omega_\tau} (u-k)_+^{p_i} |\partial_i \eta^\frac{1}{p_i}|^{p_i} \varphi \d x \d t
+ c\iint_{\Omega_\tau} \b_\alpha[u,k]\chi_{\{ u>k \}}\eta (\partial_t\varphi)_+ \d x \d t,
\end{align*}
Note that we can replace $\b_\alpha$ on both sides with the appropriate quantity using Lemma \ref{estimates:boundary_terms} (i). We can also estimate the right-hand side upwards by replacing $\tau$ by $T$. Since both terms on the left-hand side are nonnegative, we can bound each term on the left-hand side individually by the right-hand side. Taking the supremum over $\tau \in [0,T]$ and adding these estimates we end up with \eqref{est:energy}. In the case of the negative part, one instead chooses $f(s) = -(s-k)_-$ and the rest of the proof is analogous. 
\end{proof}

\begin{defin}[De Giorgi classes $\DD(p_i,\alpha, \Omega_T)$] \label{def: DG1} We define $\DD(p_i,\alpha, \Omega_T)$ to be the set of measurable functions $u: \Omega_T \rightarrow \R$ such that 
\[u \in L^{\bf{p}}_{\textrm{loc}}(0,T; W^{1,\bf{p}}_{\textrm{loc}}(\Omega))\cap L^{P}_{\textrm{loc}}(\Omega_T), 
\]
for $m$ as in \eqref{extra_integrability}, which satisfies the inequalities \eqref{general-formula} for each such $f,g:\R \rightarrow \R$.
\end{defin}

\begin{defin}[De Giorgi classes $\A(p_i,\alpha, \Omega_T)$] \label{def: DG2} We define $\A(p_i,\alpha,\Omega_T)$ to be the set of measurable functions $u: \Omega_T \rightarrow \R$ such that 
\[u \in L^{\bf{p}}_{\textrm{loc}}(0,T; W^{1,\bf{p}}_{\textrm{loc}}(\Omega))\cap L^{P}_{\textrm{loc}}(\Omega_T), \quad \quad  
\]
for $m$ as in \eqref{extra_integrability}, which satisfies the inequalities \eqref{est:energy} for each such $\varphi$ and $k \in \R_+$.
\end{defin}
\noindent Starting from this definition, we prove in the following sections some fine properties of functions in the aforementioned classes. Local boundedness, semicontinuity for functions in $\A(p_i,\alpha, \Omega_T)$ and global boundedness, properties of the mass and compact support for integrable weak solutions of the Cauchy problem, being in $\DD(p_i,\alpha, S_T)$. In particular Lemma \ref{lem:general-formula} shows that weak solutions to \eqref{eq:diffusion}-\eqref{cond:structure1}-\eqref{cond:structure2} in the sense of Definition \ref{def:weaksol} are elements of $\DD(p_i,\alpha,\Omega_T)$, and Lemma \ref{lem:elementary_real} shows the inclusion
 \[ \DD(p_i,\alpha, \Omega_T) \subseteq \A(p_i,\alpha, \Omega_T).\]

\section{Local Boundedness}
\label{sec: localboundedness}

In this section we prove the local boundedness of weak solutions. Our proofs are based on De Giorgi type iterations combining the energy estimates obtained in Lemma \ref{lem:Energy_Est} with the Sobolev embedding of Lemma \ref{rem:M_Troisi}.
Throughout the section we will use the space-time cylinders defined in Section \ref{sec: notation}, which turn out to be convenient in our setting.

As mentioned in the introduction, there are two ranges for $\bar p$ which require somewhat different proofs, and in the range corresponding to small values of $\bar p$ we also require some extra integrability of the solution. For clarity the two cases have been presented in separate subsections.

\subsection{The case $\bar p > \frac{N(\alpha +1)}{N+ \alpha +1}$}
In this section we focus on the case in which $\bar p$ satisfies the following lower bound:
\begin{align}\label{bar-p_lower_bnd}
 \bar p > \frac{N(\alpha + 1)}{N + \alpha +1}.
\end{align}
The proof of local boundedness in Theorem \ref{theo:local_bdd_large_bar_p} below is somewhat different in the cases $p_N\geq \alpha + 1$ and $p_N< \alpha + 1$, where $p_N$ indicates the largest of the exponents $p_i$. In fact, in the former case the condition \eqref{bar-p_lower_bnd} is not explicitly used in the argument, however  in this particular case \eqref{bar-p_lower_bnd} must necessarily be true due to the lower bound for the parameters $p_i$.

\begin{theo}\label{theo:local_bdd_large_bar_p}
Let $u\in \A(p_i,\alpha, \Omega_T)$ and suppose that \eqref{bar-p_lower_bnd} holds. Then $u$ is locally bounded and for any cylinder of the form $Q_{r}(x_o,t_o)$ compactly contained in $\Omega_T$ and every $\sigma \in (0,1)$ we have the explicit bound
\begin{align}\label{est:local_bddness}
 \esssup_{Q_{\sigma r}(x_o,t_o)} u \leq \max \Big\{1, c \Big((1-\sigma)^{-\frac{p_N}{\bar p}(N+\bar p)}\dashiint_{Q_{r}(x_o,t_o)} u_+^{P}\d x \d t \Big)^\frac{\bar p}{N(\bar p(1+\frac{\alpha+1}{N})- P)} \Big\},
\end{align}
where $P = \max\{\{p_s\}^N_{s=1}\cup \{\alpha +1\}\}$ and $c$ is a constant depending only on $\Lambda, N, {\bf p}, \alpha$. An analogous lower bound in terms of the $L^{P}$-norm of $u_-$ holds for the essential infimum.
\end{theo}
\begin{proof}[Proof]
Define numbers
\begin{align*}
r_j := r(\sigma + (1-\sigma)2^{-j}),
\hspace{15mm} k_j:=k(1-2^{-j})^\frac{2}{\alpha+1},
\end{align*}
where $k\geq 1$ is a number to be fixed later. We also define the cylinders 
\begin{align*}
Q_j := Q_{r_j}(x_o,t_o) = K_j \times T_j.
\end{align*}
Here, $K_j$ denotes the rectangular region in space, and $T_j$ denotes the time interval. Choose functions $\eta^s_j \in C^\infty_o(x_o^s +(-r^\frac{1}{p_s}_{j}, r^\frac{1}{p_s}_{j}); [0,1])$ such that $\eta^s_j = 1$ on $x_o^s +(-r^\frac{1}{p_s}_{j+1},r^\frac{1}{p_s}_{j+1})$ and
\begin{align*}
 |{\eta^s_j}'|^{p_s} \leq c r^{-1} (1-\sigma)^{-p_N}2^{p_Nj}.
\end{align*}
Furthermore, we define
\begin{align*}
 \eta_j(x):= \prod^N_{s=1} \eta^s_j(x_s)^{p_s}.
\end{align*}
Similarly we take $\psi_j \in C^\infty(\R;[0,1])$ such that $\psi_j=1$ on $(t_o-\tau_{j+1}, t_o]$ and $\psi_j(t) = 0$ for $t \leq t_o -\frac12(\tau_j + \tau_{j+1})$ and
\begin{align*}
 |\psi_j'| \leq c r^{-1} (1-\sigma)^{-1}2^j.
\end{align*}
Furthermore, we set
\begin{align*}
 \theta &:= \frac{2}{\alpha + 1}\max \{ p_N, \alpha +1\} = 
 \begin{cases}
\frac{2 p_N}{\alpha +1}, & p_N\geq \alpha + 1 \\
2, & p_N< \alpha +1 
\end{cases}
\\
 \gamma &:= \frac{2 \bar p}{\theta}\Big(\frac1N + \frac{1}{\alpha+1}\Big).
\end{align*}
In the case $p_N\geq \alpha + 1$ the upper bound of the exponents $p_i$ in \eqref{parameter-range} guarantees that $\gamma > 1$. In the case $p_N< \alpha +1$ we can use the lower bound \eqref{bar-p_lower_bnd} for $\bar p$ to deduce that $\gamma > 1$. We define the sequence
\begin{align*}
Y_j:=\iint_{Q_j}\big(u^\frac{\alpha+1}{2}-k_j^\frac{\alpha+1}{2}\big)_+^\theta \d x \d t,
\end{align*}
for $j \in \N_0$. Note that $Y_j$ is finite for every $j$ since $u\in L^P_\textrm{loc}(\Omega_T)$. We may use H\"older's inequality to estimate
\begin{align}\label{Y-jplusone_estim1}
Y_{j+1}\leq \Big[\iint_{Q_{j+1}} \big(u^\frac{\alpha+1}{2}-k_{j+1}^\frac{\alpha+1}{2}\big)_+^{\theta \gamma}\d x \d t\Big]^\frac{1}{\gamma}|Q_{j+1}\cap \{u>k_{j+1}\}|^{1 - \frac{1}{\gamma}}.
\end{align}
We will use the shorthand notation
\begin{align}\label{def-est:phi}
\phi:= \big(u^\frac{\alpha+1}{2}-k_{j+1}^\frac{\alpha+1}{2}\big)_+ \leq (u-k_{j+1})_+^\frac{\alpha + 1}{2},
\end{align}
where the last estimate follows from Lemma \ref{lem:elementary_real} and the fact that $2/(\alpha+1)>1$. In the following calculation we estimate the integral in \eqref{Y-jplusone_estim1} by using the relation between $\gamma$ and $\theta$, and apply H\"older's inequality with $Q=N/(N-\bar p)$ in the integral over space. We estimate one of the integrals in space by its supremum for all times in $T_{j+1}$, use the estimate in \eqref{def-est:phi} in the other integral and introduce the functions $\eta_j$ and $\psi_j$ while expanding the set of integration, which allows us to use Remark \ref{rem:M_Troisi}. Finally, we estimate both of the resulting factors using the energy estimate  \eqref{est:energy} and use the bounds for the functions $\eta_j$, $\psi_j$ and their derivatives. All in all, the calculation takes the form
\begin{align}
 \notag \iint_{Q_{j+1}} &\phi^{2\bar p (\frac1N + \frac{1}{\alpha+1})} \d x \d t \leq \int_{T_{j+1}} \Big[ \int_{K_{j+1}} \phi^{\frac{2\bar p}{N}Q'}\d x\Big]^\frac{1}{Q'}\Big[ \int_{K_{j+1}} \phi^{\frac{2\bar p}{\alpha+1}Q}\d x\Big]^\frac{1}{Q}\d t 
 \\
 \label{lcalc} &= \int_{T_{j+1}} \Big[ \int_{K_{j+1}} \phi^{2}\d x\Big]^\frac{\bar p}{N}\Big[ \int_{K_{j+1}} \phi^{\frac{2}{\alpha+1}\bar p^*}\d x\Big]^\frac{\bar p}{\bar p^*}\d t 
 \\
 \notag &\leq \Big[ \sup_{T_{j+1}}\int_{K_{j+1}} \big(u^\frac{\alpha+1}{2}-k_{j+1}^\frac{\alpha+1}{2}\big)_+^{2}\d x\Big]^\frac{\bar p}{N} \int_{T_{j+1}} \Big[ \int_{K_{j+1}} (u-k_{j+1})_+^{\bar p^*}\d x\Big]^\frac{\bar p}{\bar p^*}\d t 
 \\
 \notag &\leq \Big[ \sup_{T_{j}}\int_{K_{j}} \big(u^\frac{\alpha+1}{2}-k_{j+1}^\frac{\alpha+1}{2}\big)_+^{2} \eta_j \psi_j \d x\Big]^\frac{\bar p}{N}\int_{T_j}\psi_j \Big[ \int_{K_{j}} [(u-k_{j+1})_+\eta_j]^{\bar p^*}\d x\Big]^\frac{\bar p}{\bar p^*}\d t
 \\
 \notag & \leq c \Big[ \sup_{T_{j}}\int_{K_{j}} \big(u^\frac{\alpha+1}{2}-k_{j+1}^\frac{\alpha+1}{2}\big)_+^{2} \eta_j \psi_j \d x\Big]^\frac{\bar p}{N} \Big[ \sum^N_{s=1} \iint_{Q_j} |\partial_s[(u-k_{j+1})\eta_j]|^{p_s} \psi_j\d x \d t\Big]
 \\
 \notag &\leq c\Big[ \iint_{Q_j} \sum^N_{s=1} (u-k_{j+1})_+^{p_s}|\partial_s \eta_j^\frac{1}{p_s}|^{p_s} \psi_j + \big(u^\frac{\alpha+1}{2}-k_{j+1}^\frac{\alpha+1}{2})_+^2 \eta_j (\partial_t \psi_j)_+\d x \d t\Big]^\frac{N+ \bar p}{N}
 \\
 \notag &\leq c \frac{2^{j\frac{p_N(N+\bar p)}{N}}}{ r^{\frac{N+\bar p}{N}}  (1-\sigma)^{\frac{p_N(N+\bar p)}{N}} } \Big[  \iint_{Q_j} \sum^N_{s=1} (u-k_{j+1})_+^{p_s} + \big(u^\frac{\alpha+1}{2}-k_{j+1}^\frac{\alpha+1}{2})_+^2 \d x \d t\Big]^\frac{N+ \bar p}{N}.
\end{align}
To treat the terms in the sum in the last integral we note that 
\begin{align}\label{est:u-k_j-diff-to_p_s}
 (u-k_{j+1})_+^{p_s}\leq u^{p_s}\chi_{\{u>k_{j+1}\}} &=(u^\frac{\alpha+1}{2} - k_j^\frac{\alpha+1}{2} + k_j^\frac{\alpha+1}{2})^{\frac{2 p_s}{\alpha+1}}\chi_{\{u>k_{j+1}\}} 
 \\
\notag &\leq c(u^\frac{\alpha+1}{2} - k_j^\frac{\alpha+1}{2})_+^{\frac{2 p_s}{\alpha+1}}\chi_{\{u>k_{j+1}\}} + c k_j^{p_s}\chi_{\{u>k_{j+1}\}} 
 \\
\notag &\leq c(u^\frac{\alpha+1}{2} - k_j^\frac{\alpha+1}{2})_+^\theta + c(1 + k_j^{p_s})\chi_{\{u>k_{j+1}\}},
\end{align}
where in the last step we use Young's inequality and the definition of $\theta$. Similarly, since $\theta \geq 2$ the last term in the integral on the last row of \eqref{lcalc} can also be treated using Young's inequality:
\begin{align*}
 \big(u^\frac{\alpha+1}{2}-k_{j+1}^\frac{\alpha+1}{2})_+^2 \leq \big(u^\frac{\alpha+1}{2}-k_{j}^\frac{\alpha+1}{2})_+^\theta + \chi_{\{u>k_{j+1}\}}.
\end{align*}
We have also used the fact that $k_j < k_{j+1}$. Applying these estimates to the terms in the last integral in \eqref{lcalc} and recalling that $k\geq 1$ we obtain
\begin{align}\label{est:phistuff}
 \iint_{Q_{j+1}} \phi^{2\bar p (\frac1N + \frac{1}{\alpha+1})}\d x \d t \leq c r^{-\frac{N+\bar p}{N}} (1-\sigma)^{-\frac{p_N(N+\bar p)}{N}} 2^{j\frac{p_N(N+\bar p)}{N}} \big( Y_j + k^{P}|Q_j\cap\{u>k_{j+1}\}| \big)^\frac{N+ \bar p}{N}.
\end{align}
We now note that 
\begin{align}\label{est:super_level_set}
 |Q_j\cap\{u>k_{j+1}\}|\leq \iint_{Q_j\cap \{u >k_{j+1}\}} \frac{\big(u^\frac{\alpha+1}{2}-k_j^\frac{\alpha+1}{2}\big)^\theta}{\big(k_{j+1}^\frac{\alpha+1}{2} - k_j^\frac{\alpha+1}{2}\big)^\theta} \d x \leq k^{-P} 2^{\theta(j+1)} Y_j.
\end{align}
Combining this estimate with \eqref{est:phistuff} and \eqref{Y-jplusone_estim1} we end up with 
\begin{align}\label{est:final-Y_jplusone}
 Y_{j+1} \leq c r^{-\frac{(N+\bar p)}{\gamma N}} (1-\sigma)^{-\frac{p_N(N+\bar p)}{\gamma N}} k^{-P(1-\frac1\gamma)} b^j Y_j^{1+\frac{\bar p}{N \gamma}},
\end{align}
where $c$ and $b$ only depend on the parameters.  We are thus in a position to apply the result in Lemma \ref{lem:fastconvg} with the choices
\begin{align*}
 C = c r^{-\frac{(N+\bar p)}{\gamma N}} (1-\sigma)^{-\frac{p_N(N+\bar p)}{\gamma N}} k^{-P(1-\frac1\gamma)}, \hspace{7mm} \delta = \frac{\bar p}{N \gamma},
\end{align*}
where $c$ is the constant from \eqref{est:final-Y_jplusone}, provided that 
\begin{align}\label{cond:Y_0}
 Y_0 \leq C^{-\frac1\delta} b^{-\frac{1}{\delta^2}} = c r^{\frac{(N+\bar p)}{\bar p}} (1-\sigma)^\frac{p_N(N+\bar p)}{\bar p} k^{\frac{N P}{\bar p}(\gamma -1)}.
\end{align}
Since 
\begin{align*}
 Y_0 \leq \iint_{Q_{r}(x_o,t_o)} u_+^{P}\d x \d t,
\end{align*}
and since $|Q_{r}(x_o,t_o)| \sim r^{\frac{(N+\bar p)}{\bar p}}$ we see that \eqref{cond:Y_0} is satisfied provided that 
\begin{align}\label{k-bound}
 k \geq c\Big((1-\sigma)^{-\frac{p_N(N+\bar p)}{\bar p}} \dashiint_{Q_{r}(x_o,t_o)} u_+^{P}\d x \d t \Big)^\frac{\bar p}{N(\bar p(1+\frac{\alpha+1}{N})- P)}.
\end{align}
for a constant $c$ depending only on the parameters. The conclusion of Lemma \ref{lem:fastconvg} is that
\begin{align*}
 \iint_{Q_{\sigma r}(x_o,t_o)}\big(u^\frac{\alpha+1}{2}-k^\frac{\alpha+1}{2}\big)_+^\theta \d x \d t\leq Y_j \to 0,
\end{align*}
which shows that the integrand above is zero almost everywhere, and hence
\begin{align*}
 \esssup_{Q_{\sigma r}(x_o,t_o)} u \leq k.
\end{align*}
Recalling that the only bounds on $k$ used in the argument are \eqref{k-bound} and $k\geq 1$, we have verified \eqref{est:local_bddness}. An analogous argument gives rise to a lower bound in terms of the $L^{P}$-norm of $u_-$.
\end{proof}

\noindent By iterating the previous result we can actually get a better bound for the essential supremum. In order to proceed we first note that by the upper bound for the exponents $p_i$ in \eqref{parameter-range} and the inequality \eqref{bar-p_lower_bnd} we have
\begin{align*}
 P > P - \frac{N}{\bar p}\Big(\bar p(1 + \frac{\alpha +1}{N}) - P\Big) =: \lambda(\alpha, N, {\bf p}). 
\end{align*}
Moreover, by the definitions of $\lambda(\alpha, N, {\bf p})$ and $P$ we see that
\begin{align*}
 \lambda(\alpha, N, {\bf p}) = P - (\alpha + 1) + N \Big( \frac{P}{\bar p} - 1\Big) \geq 0.
\end{align*}
Now we can formulate the following.
\begin{theo}\label{thm:improved_local_bddness}
 Let $u\in \A(p_i,\alpha,\Omega_T)$ and suppose that \eqref{bar-p_lower_bnd} holds. Then for any $Q_{2r}(x_o,t_o)$ compactly contained in $\Omega_T$ and any $\gamma \in (\lambda(\alpha, N, {\bf p}), P]$ we have
 \begin{align}\label{est:improved_local_bddness}
  \esssup_{Q_r(x_o,t_o)} u \leq c\Big( \Big(\dashiint_{Q_{2r}(x_o,t_o)} u_+^\gamma \d x \d t\Big)^\frac{1}{\gamma - \lambda(\alpha, N, {\bf p})} + 1\Big),
 \end{align}
 for a constant $c$ depending only on $\alpha, N, {\bf p}, \Lambda, \gamma$.
\end{theo}
\begin{proof}[Proof] Let $\gamma$ be as in the statement of the lemma. Note that the case $\gamma = P$ is contained in the previous theorem, so henceforth we may assume $\gamma < P$. We define
 \begin{align*}
  r_j := (2-2^{-j})r, \hspace{7mm} Q_j:= Q_{r_j}(x_o,t_o).
 \end{align*}
Using \eqref{est:local_bddness} with $\sigma = r_j/r_{j+1}$ we obtain
\begin{align*}
 M_j := \esssup_{Q_j} &\leq c\Big( 2^{j \frac{p_N}{\bar p}(N + \bar p)} \dashiint_{Q_{j+1}} u_+^P\d x \d t\Big)^\frac{\bar p}{N(\bar p(1+\frac{\alpha+1}{N})- P)} + 1
 \\
 &\leq c M_{j+1}^\frac{\bar p (P-\gamma)}{N(\bar p(1+\frac{\alpha+1}{N})- P)} \Big( 2^{j \frac{p_N}{\bar p}(N + \bar p)} \dashiint_{Q_{j+1}} u_+^\gamma \d x \d t\Big)^\frac{\bar p}{N(\bar p(1+\frac{\alpha+1}{N})- P)} + 1.
\end{align*}
Since $\lambda(\alpha, N, {\bf p}) < \gamma < P$, we see that the exponent of $M_{j+1}$ in the last expression belongs to the interval $(0,1)$. Thus we can use Young's inequality to obtain
\begin{align*}
 M_j \leq \varepsilon M_{j+1} + c(\varepsilon) b^j\Big(\dashiint_{Q_{2r}(x_o,t_o)} u_+^\gamma \d x \d t\Big)^\frac{1}{\gamma -\lambda(\alpha, N, {\bf p})} + 1,
\end{align*}
where we have also used the fact that $Q_{j+1}$ is contained in $Q_{2r}(x_o,t_o)$. Here we emphasize that the constant $c(\varepsilon)$ depends on $\varepsilon$ in addition to the parameters $\gamma, \alpha, N, {\bf p}$. The constant $b$ depends on $\gamma, \alpha, N, {\bf p}$. By iterating the last estimate we end up with
\begin{align*}
 M_0 \leq \varepsilon^n M_n + c(\varepsilon) \Big(\sum^{n-1}_{k=1} (\varepsilon b)^k\Big) \Big(\dashiint_{Q_{2r}(x_o,t_o)} u_+^\gamma \d x \d t\Big)^\frac{1}{\gamma -\lambda(\alpha, N, {\bf p})} + \sum^{n-1}_{k=1} \varepsilon^k.
\end{align*}
Choose $\varepsilon$ so small that $\varepsilon b < 1$. Then the sums in the last estimate converge as $n\to \infty$. Since the sequence $M_n$ is bounded due to the local boundedness of $u$, the first term vanishes in the limit $n\to \infty$, and we end up with \eqref{est:improved_local_bddness}. 
\end{proof}

\subsection{The case $\bar p \leq \frac{N(\alpha +1)}{N+ \alpha +1}$}
We now turn our attention to the range 
\begin{align}\label{bar-p-small}
 \bar p \leq \frac{N(\alpha +1)}{N+ \alpha +1},
\end{align}
and recall that we also require \eqref{extra_integrability} in this case. In previous section condition $u \in L^P(\Omega_T)$ joint with \eqref{bar-p_lower_bnd} implies \eqref{bar-p-small} and there is no need to make this assumption. As in the previous case, the argument consists of two parts. First, we obtain local boundedness without an explicit bound. Once local boundedness has been established, we follow the approach of DiBenedetto in Section V.10 of \cite{DiBene} to obtain an explicit estimate for the essential supremum.
\begin{lem} \label{lem: bdd-subcritical}
 Let $u\in \A(p_i,\alpha, \Omega_T)$, and suppose that \eqref{bar-p-small} holds and that $u$ satisfies the extra integrability condition \eqref{extra_integrability}. Then $u$ is locally bounded.
\end{lem}
\begin{proof}[Proof]
Let $(x_o,t_o) \in \Omega_T$. Choose $r$ so small that $Q_r(x_o,t_o)$ is compactly contained in $\Omega_T$ and take $\sigma \in (0,1)$. We show that $u$ is bounded from above on $Q_{\sigma r}(x_o,t_o)$. By the definition of $m$ in \eqref{extra_integrability} we have
 \begin{align}\label{def:ell}
  \ell := \frac{2m}{\alpha + 1} > 2N\Big(\frac{1}{\bar p} - \frac{1}{\alpha +1}\Big) \geq 2,
 \end{align}
where the last estimate follows from \eqref{bar-p-small}. We also define 
\begin{align*}
 q:= \frac{\ell/2 - \bar p \Big(\frac1N + \frac{1}{\alpha +1}\Big)}{\ell/2 -1} \geq 1, \hspace{7mm} \mu := \frac{\bar p \Big(\frac1N + \frac{1}{\alpha +1}\Big)}{q} \in (0, 1].
\end{align*}
The lower bound for $q$ follows from \eqref{bar-p-small} and from this we deduce the range of $\mu$. Define the cylinders $Q_j$ as in the proof of Theorem \ref{theo:local_bdd_large_bar_p}, and choose
\begin{align*}
Y_j:=\iint_{Q_j}\big(u^\frac{\alpha+1}{2}-k_j^\frac{\alpha+1}{2}\big)_+^2 \d x \d t , \hspace{7mm} k_j:=k(2-2^{-j})^\frac{2}{\alpha+1}.
\end{align*}
We require $k\geq 1$ as before. Define the functions $\eta^s_j$ and $\psi_j$ and $\phi$ as in the proof of Theorem \ref{theo:local_bdd_large_bar_p}. If strict inequality holds in \eqref{bar-p-small} then $q>1$ and we have by H\"older's inequality
\begin{align}\label{est:Y_jplusone-small_bar_p}
 Y_{j+1} &= \iint_{Q_{j+1}} \big(u^\frac{\alpha+1}{2}-k_{j+1}^\frac{\alpha+1}{2}\big)_+^2 \d x \d t = \iint_{Q_{j+1}} \phi^2 \d x \d t = \iint_{Q_{j+1}} \phi^{2\mu} \phi^{2(1-\mu)}\d x \d t
 \\
 \notag &\leq \Big(\iint_{Q_{j+1}} \phi^{2\mu q}\d x \d t\Big)^\frac1q \Big(\iint_{Q_{j+1}} \phi^{2(1-\mu)\frac{q}{q-1}} \d x \d t\Big)^\frac{q-1}{q}
 \\
\notag & = \Big(\iint_{Q_{j+1}} \phi^{2\bar p (\frac1N + \frac{1}{\alpha +1})}\d x \d t\Big)^\frac1q \Big(\iint_{Q_{j+1}} \phi^\ell \d x \d t\Big)^\frac{q-1}{q}
\\
\notag &\leq \Big(\iint_{Q_{j+1}} \phi^{2\bar p (\frac1N + \frac{1}{\alpha +1})}\d x \d t\Big)^\frac1q \Big(\iint_{Q_0} u_+^m \d x \d t\Big)^\frac{q-1}{q}.
\end{align}
If instead \eqref{bar-p-small} holds with equality then $q=1$ and the previous estimate is seen to be valid without applying H\"older's inequality. We note that the estimate \eqref{lcalc} is valid for the first integral on the last line, despite the different definition of $k_j$. The upper bound for the exponents $p_s$ in \eqref{parameter-range} together with \eqref{bar-p-small} show that $p_s < \alpha + 1$. Therefore, similarly as in \eqref{est:u-k_j-diff-to_p_s} we have
\begin{align}\label{est:rg}
 (u-k_{j+1})_+^{p_s}  &\leq c(u^\frac{\alpha+1}{2} - k_j^\frac{\alpha+1}{2})_+^2 + c(1 + k_j^{p_s})\chi_{\{u>k_{j+1}\}}.
\end{align}
Combining this with \eqref{lcalc} we obtain
\begin{align*}
 \iint_{Q_{j+1}} \phi^{2\bar p (\frac1N + \frac{1}{\alpha +1})}\d x \d t \leq c r^{-\frac{N+\bar p}{N}}  (1-\sigma)^{-\frac{p_N(N+\bar p)}{N}} 2^{j\frac{p_N(N+\bar p)}{N}} \big(  Y_j + k^{p_N} |Q_j \cap \{u > k_{j+1}\}| \big)^\frac{N+ \bar p}{N},
\end{align*}
and reasoning as in \eqref{est:super_level_set} we have 
\begin{align*}
 |Q_j \cap \{u > k_{j+1}\}| \leq 2^{2j} k^{-(\alpha + 1)} Y_j.
\end{align*}
Noting that in the current parameter range $p_N< \alpha + 1$ and recalling that $k \geq 1$ we can now combine the last three estimates to conclude that 
\begin{align*}
 Y_{j+1} \leq c r^{-\frac{N+\bar p}{qN}}  (1-\sigma)^{-\frac{p_N(N+\bar p)}{qN}} I^\frac{q-1}{q} b^j Y_j^\frac{N+\bar p}{qN},
\end{align*}
where $I$ denotes the integral over $Q_0$ appearing on the last line of \eqref{est:Y_jplusone-small_bar_p} and $b$ is a constant depending only on $N, \alpha, {\bf p}, \ell$. Note that boundedness from above is obvious if $I=0$, so we assume henceforth that $I>0$. By the definition of $q$ and the lower bound for $\ell$ we see that
\begin{align*}
 \frac{N+\bar p}{qN} = \Big(\frac{N + \bar p}{N}\Big) \frac{\ell/2 -1}{\ell/2 - \bar p \big(\frac1N + \frac{1}{\alpha +1}\big)} &= \Big(\frac{N + \bar p}{N}\Big) \Big(1 + \frac{\bar p\big(\frac1N + \frac{1}{\alpha +1}\big) -1}{\ell/2 - \bar p \big(\frac1N + \frac{1}{\alpha +1}\big)}\Big)
 \\
 &> \Big(\frac{N + \bar p}{N}\Big) \Big(1 + \frac{\bar p\big(\frac1N + \frac{1}{\alpha +1}\big) -1}{N\big(\frac{1}{\bar p} - \frac{1}{\alpha+1}\big) - \bar p \big(\frac1N + \frac{1}{\alpha +1}\big)}\Big) = 1.
\end{align*}
Thus, the sequence $(Y_j)$ satisfies the recursive estimate of Lemma \ref{lem:fastconvg} with
\begin{align*}
 C = c r^{-\frac{N+\bar p}{qN}}  (1-\sigma)^{-\frac{p_N(N+\bar p)}{qN}} I^\frac{q-1}{q}, \hspace{7mm} \delta= \frac{N+\bar p}{qN} -1 > 0. 
\end{align*}
By Remark \ref{lem:fastconvg}, the sequence $(Y_j)$ converges to zero provided that 
\begin{align*}
 \iint_{Q_r(x_o,t_o)} \big(u^\frac{\alpha+1}{2}-k^\frac{\alpha+1}{2}\big)_+^2 \d x \d t = Y_0 \leq C^{-\frac{1}{\delta}} b^{-\frac{1}{\delta^2}} = c r^\frac{N + \bar p}{N + \bar p  - qN} (1-\sigma)^\frac{p_N(N + \bar p)}{N + \bar p  - qN} I^{-\frac{N(q-1)}{N + \bar p -qN}}.
\end{align*}
The integral on the left-hand side vanishes in the limit $k\to \infty$ by the dominated convergence theorem and the expression on the right-hand side is independent of $k$, so the estimate must indeed hold for large $k$. Thus, we have
\begin{align*}
 \iint_{Q_{\sigma r}(x_o,t_o)} (u^\frac{\alpha+1}{2} - 2k^\frac{\alpha+1}{2})_+^2 \d x \d t \leq Y_j \to 0,
\end{align*}
which implies that $u\leq 2^\frac{2}{\alpha+1} k$ a.e. in $Q_{\sigma r}(x_o, t_o)$ for some sufficiently large $k$.
\end{proof}
Knowing that solutions are locally bounded we now proceed to prove an explicit upper bound.
\begin{theo}
 Let $u \in \A(p_i,\alpha,\Omega_T)$  and suppose that \eqref{bar-p-small} holds and that $u$ satisfies the integrability condition \eqref{extra_integrability}. Then $u$ is locally bounded in $\Omega_T$ and if $Q_{2r}(x_o,t_o)$ is compactly contained in $\Omega$ we have the following explicit bound:
\begin{align*}
 \esssup_{Q_r(x_o,t_o)} u \leq c \Big( \Big(\dashiint_{Q_{2r}(x_o,t_o)} u_+^m\d x \d t \Big)^\frac{1}{m - \frac{N}{\bar p}(\alpha + 1 - \bar p)} + 1\Big),
\end{align*}
for a constant $c$ depending only on $N, \alpha, {\bf p}, \Lambda$ and $m$.
\end{theo}
\begin{proof}[Proof]
 We first derive a bound on cylinders of the form $Q_{\sigma r}(x_o,t_o)$ when $\sigma \in (0,1)$ and $Q_r(x_o,t_o)$ is compactly contained in $\Omega_T$. Define the radii $r_j$, the cylinders $Q_j$ and the numbers $k_j$ as in the proof of Theorem \ref{theo:local_bdd_large_bar_p}, and let $k\geq 1$. We again let $\ell$ be defined as in \eqref{def:ell} and consider the sequence of integrals 
 \begin{align*}
  Y_{j+1}:= \iint_{Q_{j+1}} \big(u^\frac{\alpha+1}{2}-k_{j+1}^\frac{\alpha+1}{2}\big)_+^\ell \d x \d t  \leq \norm{u_+}_{L^\infty(Q_0)}^\gamma \iint_{Q_{j+1}} \big(u^\frac{\alpha+1}{2}-k_{j+1}^\frac{\alpha+1}{2}\big)_+^{2\bar p (\frac{1}{N}+\frac{1}{\alpha + 1})} \d x \d t,
 \end{align*}
where 
\begin{align*}
 \gamma := \Big(\frac{\alpha+1}{2}\Big)\Big(\ell - 2\bar p \big(\frac{1}{N}+\frac{1}{\alpha + 1}\big) \Big) > 0,
\end{align*}
and the last estimate follows from combining \eqref{def:ell} and \eqref{bar-p-small}. We note that we can again use \eqref{lcalc} to estimate the last integral expression to conclude that 
\begin{align*}
 Y_{j+1} \leq c \frac{\norm{u_+}_{L^\infty(Q_0)}^\gamma 2^{j\frac{p_N(N+\bar p)}{N}} }{r^{\frac{N+\bar p}{N}}  (1-\sigma)^{\frac{p_N(N+\bar p)}{N}}} \Big[ \iint_{Q_j} \sum^N_{s=1} (u-k_{j+1})_+^{p_s} + \big(u^\frac{\alpha+1}{2}-k_{j+1}^\frac{\alpha+1}{2})_+^2 \d x \d t\Big]^\frac{N+ \bar p}{N}.
\end{align*}
Making use of \eqref{est:rg} and reasoning as in \eqref{est:super_level_set} we end up with 
\begin{align*}
 \iint_{Q_j} \sum^N_{s=1} (u-k_{j+1})_+^{p_s} + \big(u^\frac{\alpha+1}{2}-k_{j+1}^\frac{\alpha+1}{2})_+^2 \d x \d t &\leq c 2^{2j}\iint_{Q_j\cap \{u>k_{j+1}\}} \big(u^\frac{\alpha+1}{2}-k_{j}^\frac{\alpha+1}{2})_+^2 \d x \d t
 \\
 &= c 2^{2j}\iint_{Q_j\cap \{u>k_{j+1}\}} \frac{\big(u^\frac{\alpha+1}{2}-k_{j}^\frac{\alpha+1}{2})_+^\ell}{ \big(u^\frac{\alpha+1}{2}-k_{j}^\frac{\alpha+1}{2})_+^{\ell-2}} \d x \d t
 \\
 &\leq c k^{- \frac{(\alpha+1)}{2} (\ell-2)} 2^{\ell j} Y_j.
\end{align*}
Combining the last two estimates we obtain
\begin{align*}
  Y_{j+1} \leq c \norm{u_+}_{L^\infty(Q_0)}^\gamma r^{-\frac{N+\bar p}{N}}  (1-\sigma)^{-\frac{p_N(N+\bar p)}{N}} k^{- \frac{(\alpha+1)}{2}\frac{(N+\bar p)}{N} (\ell-2)} b^j Y_j^{1+\frac{\bar p}{N}},
\end{align*}
where the constants $c$ and $b$ depend on $\alpha, N, {\bf p}, \ell$. From Lemma \ref{lem:fastconvg} we see that $(Y_j)$ converges to zero provided that 
\begin{align*}
 \iint_{Q_r(x_o,t_o)} u_+^m\d x \d t = Y_0 \leq c \norm{u_+}_{L^\infty(Q_0)}^{-\frac{N \gamma}{\bar p}} r^{\frac{N+\bar p}{ \bar p}} (1-\sigma)^{\frac{p_N(N+\bar p)}{\bar p}} k^{\frac{(\alpha+1)}{2}\frac{(N+\bar p)}{\bar p} (\ell-2)},
\end{align*}
which holds if 
\begin{align}\label{bound:k}
 k \geq c\norm{u_+}_{L^\infty(Q_r(x_o,t_o))}^\zeta (1-\sigma)^{-\frac{2p_N}{(\alpha+1)(\ell-2)}} \Big(\dashiint_{Q_r(x_o,t_o)} u_+^m\d x \d t \Big)^\frac{2 \bar p}{(\alpha+1)(N+\bar p)(\ell-2)},
\end{align}
where 
\begin{align*}
 \zeta = \frac{2N\gamma}{(N+\bar p)(\alpha+1)(\ell-2)} &= \frac{N}{(N+\bar p)(\ell-2)}\Big(\ell - 2\bar p \big(\frac{1}{N}+\frac{1}{\alpha + 1}\big) \Big)
 \\
 & = \frac{N}{N+\bar p}\Big(1 + \frac{2-2\bar p(\frac1N + \frac{1}{\alpha+1})}{\ell -2 }\Big).
\end{align*}
If \eqref{bar-p-small} holds with equality then the numerator in the last fractional expression vanishes and $\zeta = N/(N+\bar p) < 1$. If instead \eqref{bar-p-small} holds with a strict inequality then also the last inequality in \eqref{def:ell} is strict and we may estimate
\begin{align*}
 \zeta < \frac{N}{N+\bar p}\Big(1 + \frac{2-2\bar p(\frac1N + \frac{1}{\alpha+1})}{2N\big(\frac1N-\frac{1}{\alpha+1}\big) -2 }\Big)  = 1.
\end{align*}
Thus, we have confirmed that in any case $\zeta \in (0,1)$. We have showed that if \eqref{bound:k} and $k\geq 1$ hold, then
\begin{align*}
 \iint_{Q_{\sigma r}(x_o,t_o)} \big(u^\frac{\alpha+1}{2}-k^\frac{\alpha+1}{2}\big)_+^\ell \d x \d t \leq Y_j \to 0,
\end{align*}
which means that $u\leq k$ in the cylinder $Q_{\sigma r}(x_o,t_o)$. Thus,
\begin{align*}
 \norm{u_+}_{L^\infty(Q_{\sigma r}(x_o,t_o))}\leq \frac{c\norm{u_+}_{L^\infty(Q_r(x_o,t_o))}^\zeta}{ (1-\sigma)^{\frac{2p_N}{(\alpha+1)(\ell-2)}} }  \Big(\dashiint_{Q_r(x_o,t_o)} u_+^m\d x \d t \Big)^\frac{2 \bar p}{(\alpha+1)(N+\bar p)(\ell-2)} + 1.
\end{align*}
Since $\zeta \in (0,1)$, we may use Young's inequality on the right-hand side to conclude that
\begin{align*}
 \norm{u_+}_{L^\infty(Q_{\sigma r}(x_o,t_o))} &\leq \varepsilon \norm{u_+}_{L^\infty(Q_r(x_o,t_o))} 
 \\
 & \hphantom{\leq} + c(\varepsilon) (1-\sigma)^{-\frac{2p_N}{(\alpha+1)(\ell-2)(1-\zeta)}} \Big(\dashiint_{Q_r(x_o,t_o)} u_+^m\d x \d t \Big)^\frac{1}{m - \frac{N}{\bar p}(\alpha + 1 - \bar p)} + 1.
\end{align*}
where we may choose $\varepsilon$ arbitrarily small. Now an iteration similar to that performed in the proof of Theorem \ref{thm:improved_local_bddness} leads to the desired estimate. 
\end{proof}

\section{Semicontinuity and Critical Mass Lemma}

\label{sec: semicontinuity}
\noindent 
In this Section we show a measure theoretical maximum principle for functions in $\A(p_i,\alpha, \Omega_T)$, usually referred to as De Giorgi-type Lemma, or Critical Mass Lemma. As a consequence, the elements of $\A(p_i,\alpha, \Omega_T)$ are lower-semicontinuous. For convenience, we formulate the result for solutions defined in a neighborhood of the origin, but evidently the result is translation invariant.

\begin{lem}[De Giorgi-type]\label{lem:DG-type}
Suppose $u\in \A(p_i,\alpha,\Omega \times (-T,T))$ and suppose that $\rho>0$ is so small that 
\begin{align*}
 \mathcal{Q}_\rho := K_\rho \times (-\rho,\rho) = \prod_{j=1}^N (-\rho^{\frac{1}{p_j}}, \rho^{\frac{1}{p_j}}) \times (-\rho,\rho)
\end{align*}
is compactly contained in the domain of $u$. Let $a \in (0,1)$, $M>0$ and suppose that $\mu^-$ and $\mu^+$ satisfy 
\begin{align*}
 \mu^- \leq \essinf_{\mathcal{Q}_\rho} u, \quad \mu^+ \geq \esssup_{\mathcal{Q}_\rho} u.
\end{align*}
Then there exists a constant $\nu^- \in (0,1)$ depending only on $a, M, \mu^-$ and the data (and independent of $\rho$) such that 
\begin{align}\label{est: measure-above}
 |\{u \leq \mu^- + M \}\cap \mathcal{Q}_\rho| \leq \nu^- |\mathcal{Q}_\rho| \quad \Rightarrow \quad u \geq \mu^- + aM \,  \textrm{ a.e. in }  \mathcal{Q}_{\rho/2}.
\end{align}
Similarly, there exists a constant $\nu^+ \in (0,1)$ depending only on $a, M, \mu^+$ and the data (and independent of $\rho$) such that 
\begin{align}\label{est: measure-below}
 |\{u \geq \mu^+ - M \}\cap \mathcal{Q}_\rho| \leq \nu^+ |\mathcal{Q}_\rho| \quad \Rightarrow \quad u \leq \mu^+ - aM \,  \textrm{ a.e. in }  \mathcal{Q}_{\rho/2}.
\end{align}
\end{lem}

\vskip0.2cm \noindent 
By Lemma \ref{lem:DG-type}, the function $u$ satisfies property $(\mathcal{D})$ of \cite{Naian} and therefore we obtain the following pointwise behaviour as a corollary.
\begin{cor}
Let $u \in \A(p_i,\alpha, \Omega_T)$, then $u$ has a lower-semicontinuous representative.
\end{cor}
\noindent Naturally $u$ also has an upper-semicontinous representative.
The corollary implies that the corresponding semicontinuous representatives of $u$ actually satisfy the pointwise bound \eqref{est: measure-above} or \eqref{est: measure-below} for every point in the intrinsic half-cylinder.

\vskip0.2cm \noindent

\begin{proof}[Proof of Lemma \ref{lem:DG-type}]
We prove \eqref{est: measure-above}, the other case being similar. The only careful passage is that to have \eqref{est: measure-above} it is only needed that the function $u$ is bounded from below, whereas in order to prove \eqref{est: measure-below} the function must be bounded from above. In our case $u \in \A(p_i,\alpha, \Omega \times (-T,T))$ and therefore by Theorem \ref{theo:local_bdd_large_bar_p} and Lemma \ref{lem: bdd-subcritical} it is locally bounded.
\vskip0.1cm

\noindent For $n \in \N_0$ we define the following sequences of real numbers and sets:
\begin{align*}
\rho_n &:= \frac\rho2 (1+2^{-n}), \quad k_n := \mu^- + aM +(1-a)2^{-n}M, 
\\
Q_n &:= \mathcal{K}_n \times \mathcal{T}_n = \prod_{s=1}^N (-\rho^\frac{1}{p_s}_n, \rho^\frac{1}{p_s}_n) \times (-\rho_n,\rho_n), \quad A_n := \{ u < k_n\} \cap Q_n.
\end{align*}
Choose functions $\eta_n^s \in C^\infty_0((-\rho^\frac{1}{p_s}_n,\rho^\frac{1}{p_s}_n);[0,1])$ such that $\eta_n = 1$ on $(-\rho^\frac{1}{p_s}_{n+1}, \rho^\frac{1}{p_s}_{n+1})$ and $\varphi_n \in C^\infty((-\rho_n,\rho_n);[0,1])$ such that $\varphi_n = 1$ on $(-\rho_{n+1},\rho_{n+1})$ and 
\begin{align*}
 |{\eta^s_n}'| \leq c 2^n \rho^{-\frac{1}{p_s}}, \quad |\varphi_n'| \leq c 2^n \rho^{-1}.
\end{align*}
Combining the energy estimate \eqref{est:energy} with points (i) and (ii) of Lemma \ref{estimates:boundary_terms} we see that 
\begin{align*}
&\sum^N_{j=1} \iint_{Q_n} |\partial_j [(u-k_n)_- \eta_n]|^{p_j} \varphi_n \d x \d t + \sup_{\tau\in \mathcal{T}_n}\int_{\mathcal{K}_n \times \{\tau\}} (|u| + |k_n|)^{\alpha-1}(u-k_n)_-^2 \eta_n \varphi_n \d x
\\
\notag &\leq \sum^N_{j=1} c\iint_{Q_n} (u-k_n)_-^{p_j}|\partial_j \eta^\frac{1}{p_j}|^{p_j} \varphi\d x\d t + c\iint_{Q_n} (|u| + |k_n|)^{\alpha-1}(u-k_n)_-^2 \eta_n (\partial_t \varphi_n)_+\d x \d t
\\
&:= I_1 + I_2,
\end{align*}
where $\eta_n$ is defined using the functions $\eta^s_n$ as in \eqref{expr:eta}. The term $I_1$ can be estimated as
\begin{align*}
 I_1 \leq c\sum^N_{j=1} M^{p_j}\norm{{\eta^j_n}'}_\infty^{p_j}|A_n| \leq c 2^{p_N n} \rho^{-1}|A_n| \sum^N_{j=1}M^{p_j}.
\end{align*}
Note now that 
\begin{align*}
 (u-K_n)_- \leq |u - k_n| \leq |u| + |k_n|,
\end{align*}
and since also $\alpha < 1 $ we have
\begin{align*}
 (|u| + |k_n|)^{\alpha-1} \leq (u - k_n)_-^{\alpha - 1}.
\end{align*}
This observation allows us to estimate $I_2$ as 
\begin{align*}
 I_2 \leq c \iint_{Q_n} (u-k_n)_-^{\alpha+1} \eta_n (\partial_t \varphi_n)_+\d x \d t \leq c M^{\alpha+1} 2^n \rho^{-1} |A_n|.
\end{align*}
Noting that 
\begin{align*}
 |u| + |k_n| \leq 2 L :=2 \max\{|\mu^-|, |\mu^- + M|\},
\end{align*}
we can combine the energy estimate with the estimates for $I_1$ and $I_2$ to conclude that 
\begin{align}\label{est:energy-consequence}
 \sum^N_{j=1} &\iint_{Q_n} |\partial_j [(u-k_n)_- \eta_n]|^{p_j} \varphi_n \d x \d t + (2L)^{\alpha-1} \sup_{\tau\in \mathcal{T}_n}\int_{\mathcal{K}_n \times \{\tau\}} (u-k_n)_-^2 \eta_n \varphi_n \d x
\\
\notag &\leq c 2^{p_N n}\rho^{-1}|A_n|\Big( \sum^N_{j=1}M^{p_j} + M^{\alpha+1}\Big).
\end{align}
By the definitions of the sequences and sets and H\"older's inequality see that
\begin{align}\label{est:A_nplusone-EINS}
 (1-a)M 2^{-(n+1)} |A_{n+1}| &= ( k_n - k_{n+1})|A_{n+1}| 
 \\
 \notag &\leq \iint_{A_{n+1}} (u - k_n)_- \d x \d t 
 \\
\notag  &\leq \Big[ \iint_{A_{n+1}} (u - k_n)_-^{\bar p(1+\frac{2}{N})} \d x \d t \Big]^\frac{N}{\bar p (N+2)} |A_{n+1}|^{1-\frac{N}{\bar p(N+2)}}.
\end{align}
The integral in the last expression can be estimated using H\"older's inequality, the anisotropic Sobolev inequality \eqref{est:Troisi-application} and \eqref{est:energy-consequence} as
\begin{align}\label{blbr}
 \notag \iint_{A_{n+1}} (u - k_n)_-^{\bar p \frac{2}{N} + \bar p} & \d x \d t \leq \int_{\mathcal{T}_{n+1}} \Big[ \int_{\mathcal{K}_{n+1}} (u-k_n)_-^2\d x \Big]^\frac{\bar p}{N} \Big[ \int_{\mathcal{K}_{n+1}} (u - k_n)_-^{\bar p^*} \d x \Big]^\frac{\bar p}{\bar p^*} \d t
 \\
 & \leq \Big[ \sup_{\mathcal{T}_{n+1}} \int_{\mathcal{K}_{n+1}} (u-k_n)_-^2\d x  \Big]^\frac{\bar p}{N} \int_{\mathcal{T}_{n+1}} \Big[ \int_{\mathcal{K}_{n+1}} (u - k_n)_-^{\bar p^*} \d x \Big]^\frac{\bar p}{\bar p^*} \d t
 \\
 \notag &\leq  \Big[ \sup_{\mathcal{T}_{n}} \int_{\mathcal{K}_{n}} (u-k_n)_-^2\eta_n \varphi_n \d x  \Big]^\frac{\bar p}{N} \int_{\mathcal{T}_{n}} \Big[ \int_{\mathcal{K}_{n}} [(u - k_n)_- \eta_n ]^{\bar p^*} \d x \Big]^\frac{\bar p}{\bar p^*} \varphi_n\d t
  \\
 \notag &\leq c \Big[ \sup_{\mathcal{T}_{n}} \int_{\mathcal{K}_{n}} (u-k_n)_-^2\eta_n \varphi_n \d x  \Big]^\frac{\bar p}{N} \sum^N_{j=1}  \iint_{Q_n}  |\partial_j[(u - k_n)_- \eta_n]|^{p_j} \varphi_n \d x \d t
\\
\notag &\leq c (2L)^{(1-\alpha)\frac{\bar p}{N}}\Big[ 2^{p_N n}\rho^{-1}|A_n|\Big( \sum^N_{j=1}M^{p_j} + M^{\alpha+1}\Big) \Big]^{1+ \frac{\bar p}{N} }.
\end{align}
Combining \eqref{est:A_nplusone-EINS} and \eqref{blbr} we end up with
\begin{align*}
 |A_{n+1}| \leq c (1-a)^{-1} L^\frac{1-\alpha}{N+2} \rho^{-\frac{(N+ \bar p)}{\bar p(N+2)}} M^{-1} \Big( \sum^N_{j=1}M^{p_j} + M^{\alpha+1}\Big)^{\frac{N+ \bar p}{\bar p(N+2)}} b^n |A_n|^{1+\frac{1}{N+2}},
\end{align*}
where $c$ and $b$ only depends on the data. Noting that $|Q_n|\sim \rho^\frac{N+\bar p}{\bar p}$ we see that 
\begin{align*}
 Y_{n+1} := \frac{|A_{n+1}|}{|Q_{n+1}|} \leq c (1-a)^{-1} L^\frac{1-\alpha}{N+2}  M^{-1} \Big( \sum^N_{j=1}M^{p_j} + M^{\alpha+1}\Big)^{\frac{N+ \bar p}{\bar p(N+2)}} b^n Y_n^{1+\frac{1}{N+2}}.
\end{align*}
Thus, we can use Lemma \ref{lem:fastconvg} to conclude that $Y_n \to 0$ and that $u\geq \mu^- + a M$ a.e. in $\mathcal{Q}_{\rho/2}$ provided that 
\begin{align*}
 \frac{|\{u<\mu^- + M\} \cap \mathcal{Q}_\rho|}{|\mathcal{Q}_\rho|} = Y_0 &\leq c (1-a)^{N+2} L^{\alpha-1} M^{N+2} \Big( \sum^N_{j=1}M^{p_j} + M^{\alpha+1}\Big)^{-\frac{(N+ \bar p)}{\bar p}} 
 \\
 &=: \nu(a, \mu^-, M),
\end{align*}
where $c$ in the second last expression depends only on the data, and since $L$ only depends on $\mu^-$ and $M$, we see that $\nu$ has the right parameter dependence.
\end{proof}

\section{Properties of solutions to the Cauchy problem}
\label{sec: Cauchy}

In this section we prove the global boundedness of solutions to the boundary value problem
\begin{align}\label{prob:cauchy}
 \left\{
\begin{array}{ll}
\partial_t \big( |u|^{\alpha -1} u\big)  - \nabla\cdot A(x,t,u,\nabla u) = 0, & \quad \text{in } S_T:=\R^N \times (0,T), 
\\[5pt]
 u(x,0) = u_0(x),  & \quad x \in \R^N.
\end{array}
\right.
\end{align}
In order to make this precise, we present the exact definition
\begin{defin}\label{def:Lp-integrable-sol}
Let $u_0 \in L^{1+\alpha}_\textnormal{loc}(\R^N)$. We say that $u \in L^{\bf p}(0,T; W^{1, {\bf p}}_\textrm{loc}(\R^N)) \cap L^{P}(0,T; L^P_\textrm{loc}(\R^N))$ is an $L^{\bf p}$-integrable weak solution to the problem \eqref{prob:cauchy} if $u\in \cap_{i=1}^N L^{p_i}(S_T)$, $u$ is a weak solution on $S_T$ in the sense of Definition \ref{def:weaksol} and $|u|^{\alpha-1}u(\cdot,t)\rightarrow |u_0|^{\alpha-1}u_0$ in $L^{1+\frac{1}{\alpha}}_\textnormal{loc}(\R^N)$ as $t\to 0$. 
\end{defin}
Note that by Lemma \ref{lem:time-cont}, $|u|^{\alpha-1}u(\cdot, t)$ is continuous on $[0,T]$ into $L^{1+\frac{1}{\alpha}}_\textnormal{loc}(\R^N)$, so the limit of $|u|^{\alpha-1}u(\cdot,t)$ as $t\to 0$ exists for any weak solution. We can now formulate the main theorem of this section. We remark that the lower bound for $\bar p_{\alpha+1}$ in \eqref{item:global_bddd_1} is in fact equivalent to the condition \eqref{bar-p_lower_bnd}.
\begin{theo}\label{thm:global_bddness}
 Let $u_0\in L^{\alpha+1}(\R^N)$ and suppose that $u$ is an $L^{\bf p}$-integrable weak solution to \eqref{prob:cauchy} in the sense of Definition \ref{def:Lp-integrable-sol}. 
 \begin{enumerate}
  \item\label{item:global_bddd_1} If $\bar p_{\alpha+1} > \alpha+1$, then for all $\theta\in(0,T)$ and $q \in [\alpha+1, \bar p_{\alpha+1}]$ we have
 \begin{align}\label{est:sup-Lq}
  \esssup_{\R^N\times[\theta,T]} u \leq c \theta^{-\frac{(N+\bar p)}{\lambda_q}}\Big( \int^T_{\theta/2}\int_{\R^N} u_+^q\d x \d t \Big)^\frac{\bar p}{\lambda_q},
 \end{align}
 where $\lambda_q:= N(\bar p - (\alpha+1)) + q \bar p > 0$, and $c$ only depends on the data. An analogous estimate holds for the essential infimum, with $u_-$ appearing on the right-hand side.
 
 \item\label{item:global_bddd_2} If furthermore $\bar p_1 > \alpha+1$ then 
 \begin{align}\label{est:LinfL1}
 \norm{u(\cdot, \tau)}_{L^\infty(\R^N)} \leq c \tau^{-\frac{N}{\lambda_1}} \Big( \int_{\R^N} |u_0| \d x \Big)^\frac{\bar p}{\lambda_1},
 \end{align}
for all $\tau \in (0,T)$. Here we are considering the representative of $u$ for which $|u|^{\alpha-1}u$ is cotinuous with respect to time.
 \end{enumerate}

\end{theo}
Note that in order to obtain \eqref{item:global_bddd_1} in the previous theorem we are forced to assume global $L^{\alpha+1}$-integrability of the initial data $u_0$, although only local $L^{\alpha+1}$-integrability was needed to formulate the definition. We also assume that $u$ is integrable to all exponents $p_j$ on all of $S_T$ rather than just locally. By examining the proof below we will see that this assumption could be somewhat relaxed.

\skip0.5cm 

\noindent In order to proceed we first prove the following lemma which shows that the integrability properties of $u_0$ imply that also $u$ has stronger integrability than which is apparent from the definition. 

\begin{lem}\label{lem:Lalpha+1_Lalpha+1}
 Let $u$ be an $L^{\bf p}$-integrable weak solution to the Cauchy problem in the sense of Definition \ref{def:Lp-integrable-sol} with $u_0\in L^{\alpha+1}(\R^N)$. Then $u\in L^\infty(0,T; L^{\alpha+1}(\R^N))$ and 
 \begin{align}\label{est:Lalfaplusone-global}
  \norm{u(\cdot, t)}_{L^{\alpha+1}(\R^N)}\leq \norm{u_0}_{L^{\alpha+1}(\R^N)}, \hspace{7mm} t \in [0,T].
 \end{align}
\end{lem}
\begin{proof}[Proof]
Let $\xi \in C^\infty_o(\R;[0,1])$ be such that $\xi=1$ on $[-1,1]$ and $\xi = 0$ on $\R\setminus (-2,2)$. Define $\eta_j(s) = \xi(s/r^\frac{1}{p_j})$ and let $\eta\in C^\infty_o(\R^N)$ be defied according to \eqref{expr:eta}. Then $\eta=1$ on $K_r$ and $\eta=0$ outside of $K_{2r}$. Let $\varphi \equiv 1$. Then  \eqref{general-formula} with $\tau_1 = 0$, $\tau_2 = \tau$, $f(s) = s_+$ and 
\begin{align*}
 G(\tau) = \int^\tau_0 g(s) \d s,
\end{align*}
imply that 
\begin{align}\label{est:u_to_alpha_plusone_positive}
 \int_{\R^N \times\{\tau\}} u_+^{\alpha+1}\eta\d x \leq c \sum^N_{i=1} \iint_{S_T} u_+^{p_i}|\partial_i \eta^\frac{1}{p_i}|^{p_i} \d x \d t + \int_{\R^N} (u_0)_+^{\alpha+1} \eta \d x,
\end{align}
for all $\tau \in [0,T]$.  Note that we were able to omit the first term from the left-hand side of \eqref{general-formula} since it is nonnegative. Also, the third term on the right-hand side of \eqref{general-formula} vanishes since $\varphi$ is constant.

 Using the definition of $\eta$ we get an upper bound for the magnitude of the partial derivatives on the right-hand side of \eqref{est:u_to_alpha_plusone_positive}. Taking into account also the fact that $\eta = 1$ on $K_r$ we end up with
\begin{align*}
 \int_{K_r \times\{\tau\}} u_+^{\alpha+1} \d x \leq c r^{-1}\sum^N_{i=1} \iint_{S_T} u_+^{p_i} \d x \d t + \int_{\R^N} (u_0)_+^{\alpha+1} \d x,
\end{align*}
for all $\tau \in [0,T]$. Finally, taking the limit $r\to \infty$, the sum on the right-hand side vanishes since $u\in L^{p_i}(S_T)$ for every $i\in \{1,\dots,N\}$, and we obtain an estimate like \eqref{est:Lalfaplusone-global} with $u_+$ replacing $u$ and $(u_0)_+$ instead of $u_0$. The negative part $u_-$ can be treated in an analogous way and together these estimates imply \eqref{est:Lalfaplusone-global}.
\end{proof}

\vskip0.1cm \noindent In order to prove \eqref{est:LinfL1} in Theorem \ref{thm:global_bddness} we need the following lemma.
\begin{lem}\label{lem:L1L1} Let $u$ be an $L^{\bf p}$-integrable weak solution to the Cauchy problem \eqref{prob:cauchy} in the sense of Definition \ref{def:Lp-integrable-sol} with initial value $u_0$. Then 
\begin{align*}
 \norm{u(\cdot,t)}_{L^1(\R^N)}\leq \norm{u_0}_{L^1(\R^N)}
\end{align*}
for all $t\in [0,T)$. 

\end{lem}
\begin{proof}[Proof]
Choose in Lemma \ref{lem:general-formula} with $\Omega=\R^N$ the functions 
\[f_{\varepsilon}(s)= (s^2+\varepsilon)^{-\alpha/2}s, \qquad g_{\varepsilon}(s)= (|s|^{2/\alpha}+\varepsilon)^{-\alpha/2}|s|^{1/\alpha-1}s,\]
and define
\[G(s)=\int_0^s g(\tau)\, \d \tau.\]

\noindent The functions $f$ clearly satisfies the assumptions of Lemma \ref{lem:general-formula} for any fixed $\varepsilon>0$.
We take $\eta$ and $\varphi$ as in Lemma \ref{lem:Lalpha+1_Lalpha+1}. Then, as in Lemma \ref{lem:Lalpha+1_Lalpha+1} we have 
\begin{align}\label{est:almost_L1}
 \int_{\R^N} \eta G(|u|^{\alpha -1}u)(x,\tau) \d x  &\leq  \gamma \sum^N_{i=1} \iint_{S_T} |f(u)|^{p_i} f'(u)^{1-p_i} |\partial_i \eta^\frac{1}{p_i}|^{p_i} \d x \d t
 \\
 \notag & \quad +  \int_{\R^N} \eta G(|u_0|^{\alpha -1}u_0) \d x.
\end{align}
By a direct calculation we see that 
\begin{align*}
 f'(u) = [(1-\alpha)u^2 + \varepsilon ] (u^2 + \varepsilon)^{-\frac\alpha2 -1} \geq (1-\alpha)(u^2 + \varepsilon)^{-\frac\alpha2},
\end{align*}
and furthermore that
\begin{equation}
\begin{aligned}\label{est:f-mess}
 |f(u)|^{p_i}f'(u)^{1-p_i} &\leq (1-\alpha)^{1-p_i} (u^2-\varepsilon)^{-\frac{\alpha}{2}}|u|^{p_i} \leq c\varepsilon^{-\frac\alpha2} |u|^{p_i}.
\end{aligned}
\end{equation}
Using \eqref{est:f-mess} and the definition of $\eta$ we obtain
\begin{align*}
 \iint_{S_T} |f(u)|^{p_k} f'(u)^{1-p_k} |\partial_k \eta^\frac{1}{p_k}|^{p_k} \d x \d t \leq c \varepsilon^{-\frac\alpha2} r^{-1}\int^T_0 \int_{K_{2r}} |u|^{p_k}\d x \d t \xrightarrow[r\to\infty]{} 0.
\end{align*}
Thus, passing to the limit $r\to \infty$ in \eqref{est:almost_L1} we end up with
\begin{align}\label{est:G-G}
 \int_{\R^N} G(|u|^{\alpha -1}u)(x,\tau_2) \d x  &\leq  \int_{\R^N} G(|u_0|^{\alpha -1}u_0) \d x.
\end{align}
Using the definition of $G$ we see that for any $\tau \in [0,T]$,
\begin{align*}
 \int_{\R^N} G(|u|^{\alpha -1}u)(x,\tau) \d x &= \int_{\R^N} \int^{|u|^{\alpha -1}u(x,\tau)}_0 (|s|^\frac{2}{\alpha} + \varepsilon)^{-\frac{\alpha}{2}}|s|^{\frac1\alpha - 1}s \d s \d x 
 \\
 &= \int_{\R^N} \int^{|u(x,\tau)|^{\alpha}}_0 (s^\frac{2}{\alpha} + \varepsilon)^{-\frac{\alpha}{2}}s^{\frac1\alpha} \d s \d x 
 \\
 \xrightarrow[\varepsilon\to 0]{} & \int_{\R^N} \int^{|u(x,\tau)|^{\alpha}}_0 s^{\frac1\alpha - 1} \d s \d x 
 \\
 &= \alpha \int_{\R^N} |u(x,\tau)|\d x.
\end{align*}
where we make use of the monotone convergence theorem in order to pass to the limit $\varepsilon\to 0$. Thus passing to the limit $\varepsilon\to 0$ in \eqref{est:G-G} and recalling that $\tau_2$ is arbitrary gives the claim.
\end{proof}

\vskip0.1cm \noindent Now we can prove the main theorem of this section.
\begin{proof}[Proof of Theorem \ref{thm:global_bddness}] We first prove \eqref{item:global_bddd_1}. 
 Let $q\in [\alpha+1, \bar p_{\alpha+1}]$. Then $\lambda_q>0$ due to the condition $\bar p_{\alpha+1} > \alpha+1$. For $j\in \N_0$ and $k>0$ we define
 \begin{align*}
  \theta_j &:= \theta(1 - 2^{-(j+1)}), \hspace{7mm} \psi_j(t) := \min \{1, 2^{j+2}\theta^{-1}(t-\theta_j)_+\}, 
  \\
  S_j &:= \R^N \times (\theta_j,T), \hspace{7mm} k_j:= k(1-2^{-j})^\frac{2}{\alpha+1}.
 \end{align*}
We define the sequence 
\begin{align*}
 Y_j : = \iint_{S_j} \big( u^\frac{\alpha+1}{2} - k_j^\frac{\alpha+1}{2}\big)_+^\frac{2 q}{\alpha + 1}\d x \d t.
\end{align*}
Letting $r>0$ and using the function $\eta$ defined in the proof of Lemma \ref{lem:Lalpha+1_Lalpha+1} we can reason as in \eqref{lcalc} and obtain
\begin{align*}
 &\int^T_{\theta_{j+1}} \int_{K_r}  \big( u^\frac{\alpha+1}{2} - k_{j+1}^\frac{\alpha+1}{2}\big)_+^{2\bar p(\frac{1}{\alpha+1}+\frac{1}{N})}\d x \d t 
 \\
 &\leq \Big[\sup_{[\theta_{j+1},T]} \int_{K_r}\big( u^\frac{\alpha+1}{2} - k_{j+1}^\frac{\alpha+1}{2}\big)_+^2\d x\Big]^\frac{\bar p}{N} \int^T_{\theta_{j+1}} \Big( \int_{K_r} (u-k_{j+1})_+^{\bar p^*}\d x\Big)^\frac{\bar p}{\bar p^*}\d t
 \\
 &\leq \Big[\sup_{[\theta_{j},T]} \int_{K_{2r}}\big( u^\frac{\alpha+1}{2} - k_{j+1}^\frac{\alpha+1}{2}\big)_+^2\psi_j\eta \d x\Big]^\frac{\bar p}{N} \int^T_{\theta_j}\Big( \int_{K_{2r}} [(u-k_{j+1})_+\eta]^{\bar p^*}\d x\Big)^\frac{\bar p}{\bar p^*}\d t
 \\
 &\leq c\Big[ \sum^N_{s=1} \int^T_{\theta_j} \int_{K_{2r}}(u-k_{j+1})_+^{p_s} |\partial_s \eta^\frac{1}{p_s}|^{p_s} \d x \d t + 2^j\theta^{-1}\int^T_{\theta_j}  \int_{K_{2r}} \big( u^\frac{\alpha+1}{2} - k_{j+1}^\frac{\alpha+1}{2}\big)_+^2 \d x \d t  \Big]^{1+\frac{\bar p}{N}}.
\end{align*}
Passing to the limit $r\to \infty$ as in the proof of Lemma \ref{lem:Lalpha+1_Lalpha+1} we are left with 
\begin{align}\label{est:high-exp}
 \iint_{S_{j+1}} \big( u^\frac{\alpha+1}{2} - k_{j+1}^\frac{\alpha+1}{2}\big)_+^{2\bar p(\frac{1}{\alpha+1}+\frac{1}{N})}\d x \d t  \leq c\Big( 2^j\theta^{-1} \iint_{S_j} \big( u^\frac{\alpha+1}{2} - k_{j+1}^\frac{\alpha+1}{2}\big)_+^2 \d x \d t  \Big)^{1+\frac{\bar p}{N}}.
\end{align}
The right-hand side is finite due to Lemma \ref{lem:Lalpha+1_Lalpha+1}. From the previous argument we can also conclude that $Y_0$ and hence every $Y_j$ is finite. Namely, it is possible to replace $k_{j+1}$ by zero, and by initially considering a smaller $\theta$ we see that \eqref{est:high-exp} implies that $u_+ \in L^{\bar p_{\alpha+1}}(S_0)$. Since also $u_+\in L^{\alpha+1}(S_0)$ by Lemma \ref{lem:Lalpha+1_Lalpha+1}, we have by interpolation that $u_+ \in L^q(S_0)$ which shows that each $Y_j$ is finite.
Since $\gamma:= \bar p_{\alpha+1}/q \geq 1$ we can use H\"older's inequality combined with \eqref{est:high-exp} to show that 
\begin{align*}
 Y_{j+1} &= \iint_{S_{j+1}} \big( u^\frac{\alpha+1}{2} - k_{j+1}^\frac{\alpha+1}{2}\big)_+^\frac{2 q}{\alpha + 1}\d x \d t 
 \\
 &\leq \Big[\iint_{S_{j+1}} \big( u^\frac{\alpha+1}{2} - k_{j+1}^\frac{\alpha+1}{2}\big)_+^{2\bar p(\frac{1}{\alpha+1}+\frac{1}{N})}\d x \d t\Big]^\frac{1}{\gamma}|S_{j+1}\cap \{u> k_{j+1}\}|^{1-\frac{1}{\gamma}}
 \\
 &\leq c\Big[2^j\theta^{-1} \iint_{S_j} \big( u^\frac{\alpha+1}{2} - k_{j+1}^\frac{\alpha+1}{2}\big)_+^2 \d x \d t \Big]^{\frac{(N+\bar p)}{\gamma N} }  |S_j\cap \{u> k_{j+1}\}|^{1-\frac{1}{\gamma}}
 \\
 &\leq c\theta^{-\frac{(N+\bar p)}{\gamma N}} 2^{j\frac{(N+\bar p)}{\gamma N}} Y_j^{\frac{(\alpha+1)}{q} \frac{(N+\bar p)}{\gamma N} } |S_j\cap \{u> k_{j+1}\}|^{(1-\frac{\alpha+1}{q})\frac{(N+\bar p)}{\gamma N} + 1-\frac{1}{\gamma}}, 
\end{align*}
where in the last step we use H\"older's inequality with the exponent $q/(\alpha+1)$. The measure of the set $S_j\cap \{u> k_{j+1}\}$ can be estimated as in \eqref{est:super_level_set}:
\begin{align*}
 |S_j\cap \{u> k_{j+1}\}| \leq c 2^{j\frac{2q}{\alpha+1}}k^{-q}Y_j.
\end{align*}
Combining the last two estimates we have
\begin{align}\label{est:recursive}
 Y_{j+1} \leq c \theta^{-\frac{(N+\bar p)}{\gamma N}}  k^{-H} b^j Y_j^{1+\delta},
\end{align}
where
\begin{align*}
 H &= q\big[ (1-\tfrac{\alpha+1}{q})\tfrac{(N+\bar p)}{\gamma N} + 1-\tfrac{1}{\gamma}\big] = \tfrac{q}{\bar p} \big( \tfrac{N(\bar p - (\alpha+1)) + q \bar p}{N+\alpha+1} \big), 
 \\
 \delta &= \tfrac{(\alpha+1)}{q} \tfrac{(N+\bar p)}{\gamma N} + (1-\tfrac{\alpha+1}{q})\tfrac{(N+\bar p)}{\gamma N} - \tfrac{1}{\gamma} = \tfrac{q}{N+\alpha+1}.
\end{align*}
In principle the constants $c,b$ appearing in \eqref{est:recursive} depend on $q$, but due to the range of $q$ we are able to pick $c,b$ that ultimately only depend on the data. By Lemma \ref{lem:fastconvg}, the sequence $(Y_j)$ converges to zero provided that 
\begin{align*}
 Y_0 \leq c \theta^{\frac{(N+\bar p)}{\delta \gamma N}}k^\frac{H}{\delta},
\end{align*}
which is true iff
\begin{align}\label{est:k-lower}
 k \geq c \theta^{-\frac{(N+\bar p)}{N(\bar p - (\alpha+1)) + q \bar p}}\Big( \int^T_{\theta/2}\int_{\R^N} u_+^q\d x \d t \Big)^\frac{\bar p}{N(\bar p - (\alpha+1)) + q \bar p}.
\end{align}
Thus, if \eqref{est:k-lower} holds, we have
\begin{align*}
 \int^T_\theta\int_{\R^N} \big( u^\frac{\alpha+1}{2} - k^\frac{\alpha+1}{2}\big)_+^\frac{2 q}{\alpha + 1}\d x \d t \leq Y_j \to 0.
\end{align*}
which means that $u\leq k$ on $\R^N\times(\theta,T)$. In particular, this is true if \eqref{est:k-lower} holds with equality, which confirms \eqref{est:sup-Lq}. 

\noindent To prove \eqref{item:global_bddd_2} we note that the condition $\bar p_1 \geq \alpha+1$ is stronger than $\bar p_{\alpha+1} \geq \alpha +1$, so we may use the conclusion of \eqref{item:global_bddd_2}. Take $q\in [\alpha+1, \bar p_1)$. For $\theta \in (0,T)$ the function $v(x,t) = u(x,t+\theta)$ solves the equation on $S_{T-\theta}$ and $v\in \cap^N_{j=1} L^{p_j}(S_{T-\theta})$ so by \eqref{item:global_bddd_1} we have 
\begin{align*}
  \esssup_{\R^N\times[\tilde \theta,T-\theta]} v \leq c {\tilde \theta}^{-\frac{(N+\bar p)}{\lambda_q}}\Big( \int^{T-\theta}_{\tilde \theta/2}\int_{\R^N} v_+^q\d x \d t \Big)^\frac{\bar p}{\lambda_q}, \hspace{7mm} \tilde \theta \in (0,T-\theta).
 \end{align*}
In the previous inequality we set 
\begin{align*}
 \tilde \theta = \tau_j = (T-\theta) 2^{-j-1}, \hspace{7mm} M_j = \esssup_{[\R^N\times[\tau_j, T-\theta]} v_+,
\end{align*}
to obtain
\begin{align*}
 M_j &\leq c \tau_j^{-\frac{N+\bar p}{\lambda_q}} \Big( \int^{T-\theta}_{\tau_{j+1}}\int_{\R^N} v_+^q\d x \d t \Big)^\frac{\bar p}{\lambda_q} 
 \\
 &\leq  c (T-\theta)^{-\frac{N+\bar p}{\lambda_q}} 2^{j\frac{(N+\bar p)}{\lambda_q}} M_{j+1}^{(q-1)\frac{\bar p}{\lambda_q}}\Big( \int^{T-\theta}_0\int_{\R^N} v_+ \d x \d t \Big)^\frac{\bar p}{\lambda_q}
 \\
 &\leq \varepsilon M_{j+1} + c(\varepsilon)(T-\theta)^{-\frac{N+\bar p}{\lambda_1}}2^{j\frac{(N+\bar p)}{\lambda_1}}\Big( \int^{T-\theta}_0\int_{\R^N} v_+ \d x \d t \Big)^\frac{\bar p}{\lambda_1},
\end{align*}
where in the last step we use Young's inequality and the fact that the exponent of $M_{j+1}$ in on the second last line lies in the interval $(0,1)$ since $\bar p_1 > \alpha+1$. Iterating the last estimate we end up with
\begin{align*}
 M_0 \leq \varepsilon^n M_n + c(\varepsilon) (T-\theta)^{-\frac{N+\bar p}{\lambda_1}} \Big(\sum^n_{j=1} (2^{\frac{N+\bar p}{\lambda_1}}\varepsilon)^j \Big)\Big( \int^{T-\theta}_0\int_{\R^N} v_+ \d x \d t \Big)^\frac{\bar p}{\lambda_1}.
\end{align*}
Taking $\varepsilon$ sufficiently small and passing to the limit $n\to \infty$ we see that 
\begin{align*}
 M_0 \leq c (T-\theta)^{-\frac{N+\bar p}{\lambda_1}} \Big( \int^{T-\theta}_0\int_{\R^N} v_+ \d x \d t \Big)^\frac{\bar p}{\lambda_1}.
\end{align*}
Rephrasing the last estimate in terms of $u$ we have
\begin{align*}
  \esssup_{\R^N\times[(T-\theta)/2+\theta, T]} u_+ \leq c (T-\theta)^{-\frac{N+\bar p}{\lambda_1}} \Big( \int^{T}_\theta\int_{\R^N} u_+ \d x \d t \Big)^\frac{\bar p}{\lambda_1}.
\end{align*}
The previous argument works also if we consider $u$ on the smaller time interval $(0,\tau)$ where $\tau \in (0,T]$ and in this case the choice $\theta= \tau/2$ is valid, leading to 
\begin{align*}
  \esssup_{\R^N\times[\frac34\tau, \tau]} u_+ \leq c \tau^{-\frac{N+\bar p}{\lambda_1}} \Big( \int^{\tau}_{\tau/2}\int_{\R^N} u_+ \d x \d t \Big)^\frac{\bar p}{\lambda_1}.
\end{align*}
Similarly, we may estimate the essential supremum of $u_-$ and combining these estimates we have
\begin{align*}
  \esssup_{\R^N\times[\frac34\tau, \tau]} |u| \leq c \tau^{-\frac{N+\bar p}{\lambda_1}} \Big( \int^{\tau}_{\tau/2}\int_{\R^N} |u| \d x \d t \Big)^\frac{\bar p}{\lambda_1}.
\end{align*}
The integral over $\R^N$ can now be estimated using Lemma \ref{lem:L1L1} leading to 
\begin{align}\label{est:almost_LinfL1}
  \esssup_{\R^N\times[\frac34\tau, \tau]} |u| \leq c \tau^{-\frac{N}{\lambda_1}} \Big( \int_{\R^N} |u_0| \d x \Big)^\frac{\bar p}{\lambda_1}.
\end{align}
We show that \eqref{est:LinfL1} follows from the last estimate. If the right-hand side of \eqref{est:almost_LinfL1} is infinite, there is nothing to prove. Otherwise, denote the right-hand side as $\kappa$ and suppose on the contrary that $\esssup_{\R^N}|u(\cdot, \tau)| > \kappa$. Then we find a set $E$ of positive measure and some $\varepsilon>0$ such that $|u(\cdot, \tau)| \geq \kappa + \varepsilon$ on $E$. We may assume that $E$ is contained in a compact set and that $u(\cdot, \tau) \geq \kappa + \varepsilon$ on $E$. Then we have for $\tau' \in (\tfrac34 \tau, \tau)$
\begin{align*}
 |E|\big((\kappa+\varepsilon)^\alpha - \kappa^\alpha\big)^{1+\frac{1}{\alpha}} &= \frac{1}{\tau-\tau'}\int^\tau_{\tau'} \int_E \big((\kappa+\varepsilon)^\alpha - \kappa^\alpha \big)^{1+\frac{1}{\alpha}} \d x \d t 
 \\
 &\leq \frac{1}{\tau-\tau'}\int^\tau_{\tau'} \int_E | |u(x,\tau)|^{\alpha-1}u(x,\tau) - |u(x,t)|^{\alpha-1}u(x,t)|^{1+\frac{1}{\alpha}} \d x \d t.
\end{align*}
The last expression vanishes in the limit $\tau'\to \tau$ due to the time-condinuity of $|u|^{\alpha-1}u$ but the left-hand side is a positive constant so we have reached a contradiction.
\end{proof}

\begin{rem}
 By examining the proofs of Lemma \ref{lem:Lalpha+1_Lalpha+1}, Lemma \ref{lem:L1L1}, and Theorem \ref{thm:global_bddness} we see that one could obtain the same results with somewhat weaker assumptions on $u$. In fact, it would suffice that 
 \begin{align*}
  \lim_{r\to\infty} r^{-1}\int^T_0 \int_{K_r} |u|^{p_j} \d x \d t = 0,
 \end{align*}
for every $j \in \{1, \dots , N\}$.
\end{rem}
We mention that it is possible to adapt the method of Theorem \ref{thm:global_bddness} to obtain an equivalent of Lemma 3.1 in \cite{DiBeHer}, where the balls are replaced by rectangles in space which use the scaling of \cite{TedDeg}. Utilizing this result and the strategy of \cite{DiBeHer} might provide an alternative although technical way for proving the estimates (19) and (20) in \cite{TedDeg}. In fact, also in \cite{TedDeg} the approach of \cite{DiBeHer} is mentioned.

\section{Support of solutions}

\label{sec: supp}

\noindent In the following we show that solutions to the Cauchy problem with compactly supported initial data have compact support, with a qualitative description involving the $L_{\textnormal{loc}}^{p_i}(S_T)$-norm of the solution. Next we use the fact that solutions are compactly supported to refine the qualitative information to a quantitative one. The support is well defined, since solutions are semi-continuous. 

We obtain our estimates on the support for $L^{\bf p}$-integrable solutions, but we remark that it should be possible to extend the results even to solutions which are $L^{p_i}$-integrable in the $i$th coordinate direction only for some of the coordinates $i\in \{1,\dots,N\}$, provided that an appropriate branch of the solution is considered. See \cite{DuMoVe} for a discussion of this topic in the case $\alpha=1$. The reason why the more general case is more involved is the non-uniqueness of solutions.  A reasonable theory of existence is available (see \cite{Raviart2}, \cite{Sango}, \cite{Raviart1} for the prototype and \cite{TedDeg2} for more general initial data and \cite{Laptev} for more general operators), but in general, uniqueness fails for $L^2_{\textnormal{loc}}$ initial data, already for the $p$-Laplace equation. 
\begin{lem}\label{lem:supp-rough}
Let $u$ be an $L^{\bf p}$-integrable weak solution to the Cauchy problem \ref{prob:cauchy} in the sense of Definition \ref{def:Lp-integrable-sol}, with $u_0 \in L^{\alpha+1}(S_T)$ and $\supp (u_0) \subset \K_{R_0}$ for some $R_0>0$. Assume that for all $j =1,\dots,N$ 
\begin{equation} \label{supp-rough}
    \alpha+1<p_j<\bar{p} (1+(\alpha+1)/N)< N+\alpha+1.
\end{equation} \noindent Then there exist $d,\chi, \tilde \chi$ depending only on the data $\{N,p_i, \alpha, \Lambda\}$ specified in \eqref{rough-exp} and $\gamma>0$ depending on the data and $T$ such that for $R\ge \max\{2R_0, 1\}$ \vskip0.2cm \noindent 
\[u(\cdot, t) \equiv 0, \quad \text{in} \quad \K_{2R} \setminus \K_R \quad \forall t \in [0,\max\{\tau^*,T\}], \quad \text{where}\]
\[ \tau^*=\tau^*(u,R,T)= \gamma R^{\frac{p_1}{d}(1+\chi)} \min \Big\{    \Big[\int_0^T \int_{\K_{3R}} \sum_{i=1}^N |u|^{p_i} \, \d x \d t\Big]^{-\chi}, \, \Big[\int_0^T \int_{\K_{3R}} \sum_{i=1}^N |u|^{p_i} \, \d x \d t\Big]^{-\tilde \chi} \Big\}.\]
\end{lem}


\begin{proof}[Proof]
Let $R\geq 2R_0$ and define the domains
\[E_n= \K_{r_n}\setminus \K_{s_n}, \quad \text{for} \quad r_n=2R(1+2^{-(n+1)}), \, s_n=R(1-2^{-(n+1)}),\]
so that $E_n \subset E_{n+1}$ is an increasing sequence of sets starting from $E_0= \K_{3R}\setminus \K_{R/2}$ and ending up with $E_{\infty}= \K_{2R}\setminus \K_{R}$. Choose $\eta_n\in C_o^{\infty}(E_n; [0,1])$ as in \eqref{expr:eta} such that
\[\eta_n\equiv 1\quad \text{in} \quad E_{n+1}, \quad \text{and} \quad |\partial_i \eta_n^{1/p_i}|\leq c 2^n/R, \quad \forall i=1, \dots, N.\]
Reasoning as in the proof of Lemma \ref{lem:Lalpha+1_Lalpha+1} we obtain an equivalent of \eqref{est:u_to_alpha_plusone_positive}, where we retain the terms involving the derivatives of $u$ on the left-hand side, namely
\begin{align}\label{deashibarai}
 \int_{\R^N \times\{\tau\}} u_+^{\alpha+1}\eta_n \d x + \sum_{i=1}^N \iint_{S_\tau}  |\partial_i u_+ |^{p_i} \eta_n \varphi \, \d x \d t \leq c \sum^N_{i=1} \iint_{S_\tau} u_+^{p_i}|\partial_i \eta_n^\frac{1}{p_i}|^{p_i} \d x \d t.
\end{align}
Note also that the term involving $u_0$ which appears on the right-hand side of \eqref{est:u_to_alpha_plusone_positive} vanishes in the current case since $\eta_n=0$ on the support of $u_0$. By adding \eqref{deashibarai} and its corresponding estimate for $u_-$ we end up with
\begin{align}\label{ogoshi}
 \int_{\R^N \times\{\tau\}} |u|^{\alpha+1}\eta_n \d x + \sum_{i=1}^N \iint_{S_\tau}  |\partial_i u |^{p_i} \eta_n \varphi \, \d x \d t \leq c \sum^N_{i=1} \iint_{S_\tau} |u|^{p_i}|\partial_i \eta_n^\frac{1}{p_i}|^{p_i} \d x \d t.
\end{align}
For $i=1,\dots, N$ we apply Theorem \ref{PAS} to the compactly supported function $u\eta_n$ with 
\[\alpha_i\equiv 1, \quad \sigma = \alpha+1, \quad q_i= p_i, \quad \theta_i= \frac{(N-\bar{p})(p_i-\alpha-1)}{N(\bar{p}-\alpha-1)+\bar{p}(\alpha+1)}.\]
These parameter choices are valid since
\[\theta_i >0 \iff p_i>\alpha+1, \quad \text{and} \quad \theta_i \leq \bar{p}/\bar{p}^* \iff p_i< \bar{p} (1+(\alpha+1)/N).\]
By summing up the resulting estimates for $i=1,\dots,N$ combined with \eqref{ogoshi} and the bound for the derivatives of $\eta_n$ we get 
\begin{align}\label{piccio2}
 \notag   Y_{n+1}:=\int_0^t \int_{E_{n+1}} \sum_{i=1}^N |u|^{p_i} \, \d x \d s &\leq c \sum_{i=1}^N  t^{1-\theta_i \bar{p}^*/\bar{p}} \bigg( \sum_{j=1}^N\frac{2^{np_j}}{R^{p_j}} \int_0^t \int_{E_n} |u|^{p_j} \d x \d s      \bigg)^{(1-\theta_i)+ \theta_i \bar{p}^*/\bar{p}}
 \\
    &\leq c_T t^{d}\bigg( \frac{2^{np_N(1+\chi_{\max})}}{R^{p_1 (1+\chi)}}\bigg) \sum_{i=1}^N Y_n^{1+\chi_i},
    \end{align}
where we emphasize that the constant $c_T$ depends on $T$ and
\begin{align} \notag d=\frac{N(\bar{p}-p_N)+\bar{p}(\alpha+1)}{N(\bar{p}-\alpha-1)+\bar{p}(\alpha+1)}, \quad \chi_i &= \theta_i(\bar{p}^*/\bar{p})-1)= \frac{\bar{p} (p_i-\alpha-1)}{N(\bar{p}-\alpha-1)+\bar{p}(\alpha+1)}, 
 \\
\label{rough-exp}
\quad \tilde \chi &=\max \{\chi_i\}_i, \quad\quad \chi=\min\{\chi_i\}_i.
\end{align}
In the last step of \eqref{piccio2} we used the fact that $t/T <1$ and to conclude that
\begin{align*}
 t^{1-\theta_i \bar{p}^*/\bar{p}} = T ^{\theta_i \bar{p}^*/\bar{p}-1}\Big(\frac{t}{T}\Big)^{1-\theta_i \bar{p}^*/\bar{p}}\leq \tilde c(T) \Big(\frac{t}{T}\Big)^{\min\{1-\theta_i \bar{p}^*/\bar{p}\}} = \tilde c(T) \Big(\frac{t}{T}\Big)^d = c(T) t^d.
\end{align*}

\noindent In light of \eqref{piccio2}, we can apply Lemma \ref{lem:fastconvg-general} with the current choice of $\chi_i$s to conclude that if
\begin{equation}\label{kondizionen}
    Y_0 \leq c \min\Big\{ \Big( \frac{t^d}{R^{p_1(1+\chi)}} \Big)^{-1/\chi},\, \Big(\frac{t^d}{R^{p_1(1+\chi)}} \Big)^{-1/\tilde \chi}\Big\}
\end{equation} 
\noindent then $\lim_{n\to \infty} Y_n=0$. Since
\begin{align*}
 Y_0 \leq \int_0^T \int_{\K_{3R}} \sum_{i=1}^N |u|^{p_i}\, \d x \d s =:I(R), 
\end{align*}
the condition \eqref{kondizionen} is satisfied if 
\begin{align*}
 t \leq c R^{\frac{p_1}{d}(1+\chi)} \min \big\{I(R)^{-\chi}, I(R)^{-\tilde \chi}\big\}.
\end{align*}
Since furthermore,
\begin{align*}
 \int^t_0\int_{E_\infty} \sum_{i=1}^N |u|^{p_i}\, \d x \d s \leq Y_n,
\end{align*}
we have that  
\[u=0, \quad \text{a.e. in} \quad (0,t) \times E_{\infty}= (0,t) \times \{\K_{2R}\setminus  \K_{R}\}, \quad \forall t \, \,  \text{satisfying} \, \,  \eqref{kondizionen}.\]
\end{proof}

Now we are ready to prove the main theorem of this section.
\begin{theo} \label{thm: support}
Suppose that for all $i = 1,\dots, N$ the condition 
\begin{equation}\label{slow diffusion}
\alpha+1<p_i \leq p_N < \bar{p} ( 1+ \alpha/N)< N + \alpha
\end{equation} is satisfied. Let $u$ be an $L^{\bf p}$-integrable weak solution to the Cauchy problem \eqref{prob:cauchy} in the sense of Definition \ref{def:Lp-integrable-sol} with 
\[ u_0 \in L^{1+\alpha}(\R^N)\cap L^1(\R^N), \quad \quad \emptyset \ne \text{supp}(u_0) \subset [-R_0,R_0]^N=: \K_{R_0}.\] Then the support of $u$ evolves with the law
\begin{equation} \label{supporto}
\text{supp}(u(\cdot, t)) \subset \prod_{i=1}^N [-R_i(t), R_i(t)], \quad R_i(t)= 2R_0+ \gamma \|u_0\|_{L^1(\R^N)}^{\frac{\bar{p}(p_i-\alpha-1)}{\lambda_1 p_i}} t^{\frac{N(\bar{p}-p_i) + \bar{p}}{\lambda_1p_i}},
\end{equation}
where $\gamma$ depends only on the data.
\end{theo}

\begin{proof}[Proof] \hskip0.1cm
Let $\varepsilon>0$. We want to use Lemma \ref{lem:general-formula} with $f(s)= (s^2+\varepsilon^2)^{\frac{\mu-1}{2}}s$, and consequently $g(s)= (|s|^{2/\alpha} + \varepsilon^2)^{\frac{\mu-1}{2}}|s|^{\frac{1-\alpha}{\alpha}}s$, where $\mu\in (0,1)$ is to be chosen later. We also take
\begin{align*}
 G(\tau) := \int^\tau_0 g(s) \d s,
\end{align*}
and $\varphi \equiv 1$. Since  
\[
f'(u) \ge \mu (s^2+\varepsilon^2)^{\frac{\mu-1}{2}}, \quad \text{and} \quad |f(u)|^{p_i} f'(u)^{1-p_i} \leq \mu^{1-p_i} |u|^{p_i} [u^2+\varepsilon^2]^{\frac{\mu-1}{2}},
\] 
we see that \eqref{general-formula} takes the form (for $0<T^*\leq T$)
\begin{equation} \label{ciccio}
    \begin{aligned}
        \int_{\R^N}\eta(x) G(|u|^{\alpha-1}u(x, t)) & \d x \bigg|_{t=0}^{T^*} + \frac{\mu}{\gamma} \sum_{i=1}^N \iint_{S_{T^*}} (u^2+\varepsilon^2)^{\frac{\mu-1}{2}} \eta(x) |\partial_i u |^{p_i}\, \d x \d t \\
        & \leq \gamma \mu^{1-p_N} \iint_{S_{T^*}} (u^2+\varepsilon^2)^{\frac{\mu-1}{2}} \sum_{i=1}^N |u|^{p_i}\, |\partial_i \eta^{1/p_i}|^{p_i}\, \d x \d t.
    \end{aligned}
\end{equation} 
By the monotone convergence theorem we see that 
\begin{align*}
 \int_{\R^N}\eta(x) G(|u|^{\alpha-1}u(x, t))  &\d x = \int_{\R^N}\eta(x) \int^{|u(x,t)|^\alpha}_0 (s^{2/\alpha} + \varepsilon^2)^{\frac{\mu-1}{2}} s^{\frac{1}{\alpha}} \d s \d x 
 \\
 &\xrightarrow[\varepsilon\to 0]{} \int_{\R^N}\eta(x) \int^{|u(x,t)|^\alpha}_0 s^{\frac{\mu}{\alpha}} \d s \d x = \tfrac{\alpha}{\alpha+\mu} \int_{\R^N}\eta(x) |u(x,t)|^{\mu+\alpha} \d x.
\end{align*}
Similarly, we can also pass to the limit $\varepsilon\to 0$ in the other terms of \eqref{ciccio} which leads to
    \begin{align}\label{ciccio2}
        \notag \frac{\alpha}{\alpha+\mu}\int_{\R^N}\eta(x) |u(x,t)|^{\mu+\alpha} \d x \bigg|_{t=0}^{T^*} +  \frac{\mu}{\gamma} \sum_{i=1}^N &\iint_{S_{T^*}}  \eta(x) |u|^{\mu-1} |\partial_i u |^{p_i} \d x \d t  \\
        & \leq \gamma \mu^{1-p_N} \iint_{S_{T^*}}  \sum_{i=1}^N |u|^{p_i+\mu-1}\, |\partial_i \eta^{1/p_i}|^{p_i} \d x \d t.
    \end{align}
Now we use a trick of \cite{DuMoVe} by wise choice of test functions. Choose an index $j \in \{1,\dots, N\}$ and let $\eta(x)=\eta_j^{p_j}(x_j) \prod_{i\ne j} \eta_i^{p_i}(x_i)  \in C_o^{\infty}(\R^N, [0,1])$ in the following way:
\[
\eta_j \equiv 0\quad \text{on} \quad (-R_0,R_0), \qquad \text{while} \qquad \eta_i \equiv 1 \quad \text{on} \quad (-R,R), \quad \text{and} \quad |\eta_i'|\leq c/R, \text{ for } i\neq j.
\] 
Note that by the properties of $u_0$ and $\eta_j$ we have that the term on the left-hand side of \eqref{ciccio2} corresponding to $t=0$ vanishes.
We let $R\rightarrow \infty$ so that all except the $j$-th term in the sum on the right-hand side of \eqref{ciccio2} vanish:
\begin{equation} \begin{aligned}
\iint_{S_{T^*}}  &\sum_{i=1}^N |u|^{p_i+\mu-1}\, |\partial_i \eta^{1/p_i}|^{p_i} \d x \d t 
\\
& \leq \sum_{i=1, i\ne j}^N \frac{C}{R^{p_i}} \iint_{S_{T^*}}   |u|^{p_i+\mu-1} \d x \d t + \iint_{S_{T^*}} |u|^{p_j+\mu-1} |\eta_j'(x_j)|^{p_j} \d x \d t
\\
&\quad \xrightarrow[R\to\infty]{} \quad \iint_{S_{T^*}} |u|^{p_j + \mu-1} |\eta_j'(x_j)|^{p_j} \d x \d t,
\end{aligned} \end{equation}
where we use the dominated convergence theorem to pass to the limit. We also use that $u \in L^1(S_T)$ by Lemma \ref{lem:L1L1} joint with $u \in \cap_{i=1}^N L^{p_i}(S_T)$ and interpolation to say that $u \in L^{p_i+\mu-1}(S_T)$. Note that this argument requires that $1\leq p_i +\mu -1 \leq p_i$. The upper bound is obvious since $\mu<1$. The lower bound follows from the first inequality of \eqref{slow diffusion} if we require $\mu \geq 1-\alpha$. This is the first bound imposed on $\mu$. 
This leads us, by generality of $T^*$, to the estimate 
\begin{equation} \label{ciccio3}
    \begin{aligned}
        \frac{\alpha}{\alpha+\mu}\sup_{t \in [0,T]}\bigg(\int_{\R^N} \eta_j^{p_j}(x_j) |u|^{\mu+\alpha}(x,t) & \d x\bigg) + \frac{\mu}{\gamma}\sum_{i=1}^N
 \iint_{S_T} \eta_j(x_j)^{p_j} |u|^{\mu-1} |\partial_i u|^{p_i} \d x \d t    
 \\
 & \leq \gamma \mu^{1-p_N}\iint_{S_T} |u|^{p_j+\mu-1} |\eta_j'(x_j) |^{p_j} \d x \d t.
 \end{aligned}
\end{equation}
\noindent  
Now we set the iterative geometry that will permit an argument \'a la De Giorgi: Define for $\rho > 2 R_0$ and $n \in \mathbb{N}$ 
\[
\rho_n= \rho (2+ 2^{-n}), \qquad s_n= \rho(1-2^{-(n+1)}), \quad E_n= \{ x \in \R^N: \, \,  s_n \leq |x_j| \leq \rho_n \},
\] and choose test functions $\eta_{j,n}\in C^\infty_o(\R;[0,1])$ with the following behavior:
\[
\eta_{j,n} \equiv 1 \quad \text{on} \quad [s_{n+1},\rho_{n+1}], \qquad |\eta_{j,n}'| \leq C2^n/\rho, \qquad  \eta_{j,n} \in C_o^{\infty} ((s_n,\rho_n)). 
\] 
Due to the definition of $s_n$ we see that $\eta_{j,n}$ vanishes on $[-R_0,R_0]$ so it is a valid test function in \eqref{ciccio3}. The sequence of sets $E_n$ is shrinking monotonically $E_{n+1} \subset E_n$ 
\[ \text{from} \quad  E_0= \{x \in \R^N: \, \,  \rho/2 \leq |x_j| \leq 3 \rho\} \quad \text{to the set} \quad E_{\infty}=\{x \in \R^N: \, \,  \rho \leq |x_j| \leq 2\rho  \}. \]  
In order to perform a De Giorgi type iteration, we want to combine \eqref{ciccio3} with Theorem \ref{PAS}. This is possible although the functions $x\mapsto \eta_j(x_j)$ are not compactly supported since $u$ is compactly supported by Lemma \ref{lem:supp-rough}.
We apply Theorem \ref{PAS} with the choices
\[
\sigma= \mu+\alpha, \qquad q= p_j+\mu-1, \quad \text{and so} \quad  \theta=(N-\bar{p}) (p_j-\alpha-1)/ [N(\bar{p}-\alpha-1)+\bar{p}(\mu+\alpha)],
\] 
with anisotropies
\[
\alpha_i= \frac{p_i+\mu-1}{p_i}, 
\] 
In order to verify that these parameter choices are valid, we note that $\theta \in (0, \bar p/\bar p^*)$ where the lower bound follows from the first inequality in \eqref{slow diffusion} and the upper bound follows from the third inequality. In order to apply Theorem \ref{PAS} we also need that $1\leq \sigma \leq p^*_\alpha$. The lower bound follows from the condition $\mu \geq 1-\alpha$ which we have already imposed. With our choice of the numbers $\alpha_i$ the upper bound can be written in the form
\begin{align}\label{popopark}
 \mu + \alpha \leq \bar p^*\Big(1 + \frac{\mu -1}{\bar p}\Big) = \frac{N}{N-\bar p}(\bar p -1 + \mu )
\end{align}
Since $\bar p > 1 + \alpha$ by \eqref{slow diffusion}, we see that \eqref{popopark} holds for all $\mu \in (0,1)$.
Thus the parameter choices are valid if we take  $\mu\geq 1-\alpha$. When applying Theorem \ref{PAS} we will use the fact that 
\begin{align}\label{eeesnow}
 \partial_i |\eta_{j,n}^{p_j} (x_j) u|^{\alpha_i} &= \alpha_i |\eta_{j,n}^{p_j}(x_j) u|^{\alpha_i -1} \partial_i|\eta_{j,n}^{p_j}(x_j) u| 
 \\
 \notag &= \alpha_i |\eta_{j,n}^{p_j}(x_j) u|^{\alpha_i -1}\Big( \delta_{ij} p_j\eta_{j,n}^{p_j-1}(x_j)\eta_{j,n}'(x_j)|u| +  \eta_{j,n}^{p_j}(x_j) \sgn(u) \partial_i u\Big).
\end{align}
Note that since $\alpha_i<1$ it is not obvious that the derivative can be calculated in this way in accordance with the chain rule. However, \eqref{eeesnow} can be verified by applying the chain rule for Sobolev functions to $(\varepsilon + |\eta_{j,n}^{p_j}(x_j) u|)^{\alpha_i} $ and passing to the limit $\varepsilon \to 0$, using the dominated convergence theorem combined with \eqref{ciccio3}. From \eqref{eeesnow} we obtain the estimate
\begin{align}\label{btut}
 |\partial_i |\eta_{j,n}^{p_j}(x_j) u|^{\alpha_i}|^{p_i} \leq c( \delta_{ij}|\eta_{j,n}'(x_j)|^{p_j} |u|^{p_j+\mu-1} + \eta_{j,n}^{p_j}(x_j)|u|^{\mu-1}|\partial_i u|^{p_i}),
\end{align}
where we also make use of the fact that $\alpha_j p_j = p_j+\mu-1 \geq 1$ and $0\leq \eta_{j,n}\leq 1$. Finally, applying Theorem \ref{PAS} to the function $\eta_{j,n}^{p_j}(x_j) u$, making use of \eqref{btut} and combining the resulting estimates with \eqref{ciccio3} and the upper bound for the derivatives $\eta_{j,n}'$ we obtain
\begin{equation*}
\begin{aligned}
    &Y_{n+1}:=\int_0^T \int_{E_{n+1}} |u|^{p_j+\mu-1} \d x \d t 
    \\
    &\leq \int^T_0 \int_{\R^N}|\eta_{j,n}^{p_j}(x_j)u|^{p_j+\mu-1} \d x \d t
    \\
    &\leq cT^{1-\theta \bar{p}^*/\bar{p}} \Big[\sup_{t \in [0,T]}\int_{\R^N} \eta_{j,n}^{p_j(\mu+\alpha)}(x_j) |u|^{\mu+\alpha} \d x\Big]^{1-\theta}\prod_{i=1}^{N}\Big[\iint_{S_{T}}|\partial_i|\eta_{j,n}^{p_j}(x_j)u|^{\alpha_{i}}|^{p_{i}}  \d x \d t\Big]^{\frac{\theta\,  \bar p^{*}}{N \, p_{i}}}
    \\
    &\leq cT^{1-\theta \bar{p}^*/\bar{p}} \prod_{i=1}^{N}\Big[\iint_{S_{T}}\delta_{ij}|\eta_{j,n}'(x_j)|^{p_j} |u|^{p_j+\mu-1} + \eta_{j,n}^{p_j}(x_j)|u|^{\mu-1}|\partial_i u|^{p_i}  \d x \d t\Big]^{\frac{\theta\,  \bar p^{*}}{N \, p_{i}}}
    \\
    &\quad \times \Big[\sup_{t \in [0,T]}\int_{\R^N} \eta_{j,n}^{p_j}(x_j) |u|^{\mu+\alpha} \d x\Big]^{1-\theta}  
    \\
    &\leq c T^{1-\theta \bar{p}^*/\bar{p}} \bigg(\frac{2^{n p_j}}{\rho^{p_j}} \int_0^T \int_{E_n}
    |u|^{p_j+\mu-1} \d x \d t\bigg)^{(1-\theta)+ \theta \bar{p}^*/\bar{p}}
    \\
    &= c T^{1-\theta \bar{p}^*/\bar{p}} \big(
    2^{n p_j}/\rho^{p_j}\big)^{(1-\theta)+ \theta \bar{p}^*/\bar{p}} Y_n^{1+\delta}.
\end{aligned}
\end{equation*} 
In the third step we also used the fact that $\mu + \alpha \geq 1$ and that $0\leq \eta_{j,n}\leq 1$ to decrease the exponent in the first integral. In principle there is a dependence on $\mu$ in \eqref{ciccio3} but it is stable as $\mu \in [1-\alpha,1)$ so ultimately we get no direct $\mu$-dependence in the constant.
The previous estimate combined with Lemma \ref{lem:fastconvg} with $\delta= \theta(\bar{p}^*/\bar{p}-1)\in (0,1)$ shows that the sequence $Y_n$ converges to zero if 
\begin{align}\label{Y_0-desired-bound}
Y_0= \int_0^T \int_{E_0}  |u|^{p_j+\mu-1} \d x \d t \leq c T^{\frac{\theta \bar{p}^*/ \bar{p}-1}{\theta(\bar{p}^*/\bar{p}-1)}} 
(\rho^{p_j})^{\frac{1+\theta(\bar{p}^*/\bar{p}-1)}{\theta(\bar{p}^*/\bar{p}-1)}}.
\end{align}
We obtain an upper bound for $Y_0$ with the help of  Lemma \ref{lem:L1L1} and Theorem \ref{thm:global_bddness}:
\begin{equation} \label{ciccio5}
\begin{aligned}
Y_0 &\leq \iint_{S_T} |u|^{p_j+\mu-1}  \d x \d t
\\
& \leq  \int_0^T \|u(\cdot, t)\|_{L^1(\R^N)}\, \|u (\cdot, t)\|_{L^{\infty}(\R^N)}^{p_j+\mu-2} \d t 
\\
&\leq c \|u_0\|_{L^1(\R^N)} \int_0^T \|u_0\|_{L^1(\R^N)}^{(\bar{p}/\lambda_1)(p_j+\mu-2)} t^{-N(p_j+\mu-2)/\lambda_1} \, \d t
\\
&= c \|u_0\|_{L^1(\R^N)}^{1+\frac{\bar{p}(p_j+\mu-2)}{\lambda_1}} T^{1-\frac{N(p_j+\mu-2)}{\lambda_1}}, \quad \quad \lambda_1= N(\bar{p}-\alpha-1)+\bar{p}.
\end{aligned}
\end{equation} \noindent 
Here above we must request that $\mu \in (0,1)$ is so small that
\begin{align}\label{mu-upperbound}
 N(p_j+\mu-2)/\lambda_1 <1 \quad \iff \quad 
 0 <\mu < \bar{p}(1+1/N) -p_j +1-\alpha. 
\end{align}
Note that due to \eqref{slow diffusion}, the upper bound for $\mu$ in \eqref{mu-upperbound} is larger than $1-\alpha$, so this new condition is compatible with our previous requirement $\mu \geq 1-\alpha$, and a $\mu$ which satisfies both conditions can thus be chosen. From \eqref{ciccio5} it follows that \eqref{Y_0-desired-bound} is true provided that
\[
\|u_0\|_{L^1(\R^N)}^{1+\frac{\bar{p}(p_j+\mu-2)}{\lambda_1}} T^{1-\frac{N(p_j+\mu-2)}{\lambda_1}}\leq c T^{\frac{\theta \bar{p}^*/ \bar{p}-1}{\theta(\bar{p}^*/\bar{p}-1)}} 
(\rho^{p_j})^{\frac{1+\theta(\bar{p}^*/\bar{p}-1)}{\theta(\bar{p}^*/\bar{p}-1)}},
\] 
which after some algebraic manipulations can be expressed as
\begin{equation} \label{raggio}
\rho^{p_j} \ge c \|u_0\|_{L^{1}(\R^N)}^{\frac{\bar{p}(p_j-\alpha-1)}{[N(\bar{p}-\alpha-1)+\bar{p}]}} T^{\frac{N(\bar{p}-p_j) + \bar{p}}{[N(\bar{p}-\alpha-1)+\bar{p}]}},  
\end{equation} so that finally 
\[
0=Y_{\infty}= \int_0^T \int_{E_\infty}  |u|^{p_j+\mu-1} \d x \d t \quad \Rightarrow \quad u=0 \quad \text{a.e. in} \quad \{x \in \R^N \,|\,   \rho\leq |x_j| \leq 2\rho \},
\] and we are done. It is noteworthy that the final condition \eqref{raggio} does not depend on $\mu$. Repeating the same procedure for $j = 1, \dots, N$ we obtain \eqref{supporto}. 

\end{proof}

\end{document}